\documentclass[11pt]{amsart}

\usepackage{amssymb,amsthm,amsmath,amsrefs,amsfonts}
\usepackage{dsfont}
\usepackage[all]{xy}

\numberwithin{equation}{section}

\newtheorem{theorem}{Theorem}[section]
\newtheorem{corollary}[theorem]{Corollary}
\newtheorem{lemma}[theorem]{Lemma}
\newtheorem{proposition}[theorem]{Proposition}
\newtheorem{definition}[theorem]{Definition}

\newtheorem{conjecture}[theorem]{Conjecture}

\newcommand{\K}{\mathcal{K}}
\newcommand{\N}{\mathbb{N}}
\newcommand{\Z}{\mathbb{Z}}
\newcommand{\R}{\mathbb{R}}
\newcommand{\C}{\mathbb{C}}
\newcommand{\Cu}{\mathrm{Cu}}

\newcommand{\id}{\mathrm{id}}

\newcommand{\M}{\mathrm{M}}

\newcommand{\I}{\mathrm{Ideal}}

\newcommand{\Her}{\mathrm{Her}}
\newcommand{\Lat}{\mathrm{Lat}}
\newcommand{\Ped}{\mathrm{Ped}}

\begin{document}
\title{Reduction of the dimension of nuclear C*-algebras}

\author{Luis Santiago}\address{Luis Santiago, Department of Mathematics, Fenton Hall, University of Oregon, Eugene OR 97403, USA}\email{luiss@uoregon.edu}
\date{\today}

\begin{abstract}
We show that for a large class of C*-algebras $\mathcal{A}$, containing arbitrary direct limits of separable type I C*-algebras, the following statement holds: If $A\in \mathcal{A}$ and $B$ is a simple projectionless C*-algebra with trivial K-groups that can be written as a direct limit of a system of (nonunital) recursive subhomogeneous algebras with no dimension growth then the stable rank of $A\otimes B$ is one. As a consequence we show that if $A\in \mathcal A$ and $\mathcal W$ is the C*-algebra constructed in \cite{Jacelon} then the stable rank of $A\otimes\mathcal W$ is one.  We also prove the following stronger result: If $A$ is separable C*-algebra that can be written as a direct limit of C*-algebras of the form $\mathrm{C}_0(X)\otimes \M_n$, where $X$ is locally compact and Hausdorff, then $A\otimes \mathcal W$ can be written as a direct limit of a sequence of 1-dimensional noncommutative CW-complexes.  
\end{abstract}

\maketitle

\section{Introduction}
The notion of covering dimension of a topological space has led to different dimension theories for C*-algebras; for instance, the stable rank, real rank, decomposition rank, and nuclear dimension. Each of these dimension theories have had important applications to the theory of C*-algebras. It was shown in \cite{Winter-dec-rank} and \cite{Winter-nuc-dim}  that simple separable nonelementary unital C*-algebras with finite decomposition rank or more generally with finite nuclear dimension absorb the Jiang-Su algebra tensorially. As a consequence, new classification results have been obtained for simple C*-algebras. For C*-algebras that can be written as direct limits of subhomogeneous algebras one can also associate a dimension, namely the infimum over all such direct limit decompositions of the supremum of the covering dimension of the spectrum of all the C*-algebras appearing on the given direct limit. For instance, it was shown in \cite{Gong} that for simple separable unital AH-algebras with very slow dimension growth this dimension is at most three. This result was later used in \cite{Elliott-Gong-Li}  to classify this class of C*-algebras.

In this paper we study different notions of dimension for certain C*-algebras of the form $A\otimes B$, where $B$ is simple, nuclear, either projectionless or unital with no nonzero projections but its unit, and with a specified direct limit decomposition.  We are particularly interested in two cases: the first case is when $B$ is the C*-algebra $\mathcal W$ constructed in \cite{Jacelon} and the second is when $B$ is the Jiang-Su algebra $\mathcal{Z}$. It was shown in \cite{Rordam-stable-rank} that if $A$ is a simple finite C*-algebra then the stable rank of $A\otimes \mathcal Z$ is one. However, it is not known in general what the exact value of the stable rank is when $A$ is not simple. In Proposition \ref{prop: Jiang-Su} we partially answer this question. We compute the stable rank in the case that $A$ is a commutative C*-algebra whose spectrum is a CW-complex. It is believed that the decomposition rank (in the stably finite case) and the nuclear dimension of C*-algebras of the form $A\otimes \mathcal Z$ is at most two. This has recently been confirmed in \cite{Tikuisis-Winter} when $A$ is a commutative C*-algebra.

The C*-algebra $\mathcal W$ is a simple separable nuclear C*-algebra that is stably finite, stably projectionless, has a unique tracial state, and has trivial K-groups. This algebra should be considered as a stably finite analog of the Cuntz algebra $\mathcal{O}_2$. It should play central role in the classification of projectionless C*-algebras. 
Let $A$ be a C*-algebra and let $\mathrm{T}(A)$ denote the cone of lower semicontinuous traces on $A_+$ with values in $[0, \infty]$ (note that the traces are not required to be densely finite). It has been shown in \cite{Elliott-Robert-Santiago} that $\mathrm{T}(A)$ belongs to the category of compact Hausdorff non-cancellative cones with jointly continuous addition and jointly continuous scalar multiplication.
Our main motivation for studying C*-algebras of the form $A\otimes W$ is the following conjecture of Leonel Robert:
\begin{conjecture}
If $A$ and $B$ are separable nuclear C*-algebras then 
$$\mathrm{T}(A)\cong \mathrm{T}(B)\quad\Longleftrightarrow\quad A\otimes W\otimes \mathcal{K}\cong B\otimes W\otimes \mathcal{K},$$
where the isomorphism between $\mathrm{T}(A)$ and $\mathrm{T}(B)$ is assumed to be a linear homeomorphism.
\end{conjecture} 
This conjecture has been shown to be true for AF-algebras and for $\mathcal{O}_2$-absorbing algebras. In fact, in the $\mathcal{O}_2$-absorbing case this conjecture is nothing more than Kirchberg Classification Theorem of $\mathcal{O}_2$-absorbing algebras. As a consequence of Theorem \ref{thm: mainmain} below we obtained the following result: if $A$ is a separable direct limit of homogeneous C*-algebras, then the tensor product $A\otimes \mathcal W$ is a direct limit of a sequence of 1-dimensional noncommutative CW-complexes (these are subhomogeneous algebras of 1-dimensional spectrum). This result reduces the proof of Robert's Conjecture for direct limits of homogeneous algebras to prove a classification result for direct limits of 1-dimensional noncommutative CW-complexes (or shortly NCCW-complexes). This result also implies that the decomposition rank and the nuclear dimension of $A\otimes \mathcal W$ is one. Direct limits of sequences of 1-dimensional NCCW-comples has been classified in \cite{Robert} in the case that the building blocks have trivial K$_1$-groups. Unfortunately, this classification result can not be applied directly in our case since all the building blocks obtained in the direct limit decomposition of $A\otimes \mathcal W$ have  non-trivial K$_1$-groups. Another consequence of Theorem \ref{thm: mainmain} is that if $A$ is a C*-algebra in the class $\mathcal{A}$ defined below then the stable rank of $A\otimes \mathcal{W}$ is one. In particular, by Theorem \ref{thm: type I} the stable rank of the tensor product of $\mathcal W$ with a direct limit of separable type I C*-algebras is one.

\begin{definition}
Let $\mathcal A$ be a class of C*-algebras. We say that a C*-algebra $B$ is locally contained in $\mathcal A$ if for every $\epsilon>0$ and every finite subset $F$ of $B$ there exists a C*-algebra $A\in \mathcal A$ and a *-homomorphism $\phi\colon A\to B$ such that the distance from $x$ to $\phi(A)$ is less than $\epsilon$ for every $x\in F$.
\end{definition}

\begin{definition}\label{def: A}
Let us denote by $\mathcal A$ the smallest class of C*-algebras that satisfies the following properties: 
\begin{itemize}
\item[(i)] $\mathrm{C}_0(X)\in \mathcal{A}$ for every locally compact space $X$.
\item[(ii)] If $A\in \mathcal{A}$ then $A\otimes \M_n(\C)\in \mathcal A$ for every $n\in \N$.
\item[(iii)] If $A\in \mathcal A$ then every hereditary sub-C*-algebra of $A$ belongs to $\mathcal A$. In particular, every closed two-sided ideal of $A$ belongs to $\mathcal A$.
\item[(iv)] If $A\in \mathcal A$ then every quotient of $A$ belongs to $\mathcal A$.
\item[(v)] If $A,C\in \mathcal{A}$ and if
$$0\to A\to B\to C\to 0$$
is an exact sequence of C*-algebras then $B\in \mathcal{A}$.
\item[(vi)] If $A$ is locally contained in $\mathcal{A}$ then $A\in \mathcal{A}$.
\end{itemize}
\end{definition}

The following theorem is our main result (see later in Sections 2 and 3 for the definition of a RSH$_0$-algebra and of a 1-dimensional NCCW-complex): 

\begin{theorem}\label{thm: mainmain}
Let $B$ be a direct limit of a system of RSH$_0$-algebras with no dimension growth. The following statements hold:
\begin{enumerate}
\item If $B$ has a finite number of ideals then $\mathrm{sr}(B)=1$, if and only if, for every ideal $I$ of $B$ the index map $\delta\colon \mathrm{K}_1(A/I)\to \mathrm{K}_0(I)$ is trivial.

\item  If $B$ is simple, projectionless, and $\mathrm{K}_0(B)=\mathrm{K}_1(B)=0$ then $\mathrm{sr}(A\otimes B)=1$ for every C*-algebra $A\in \mathcal{A}$.

\item If $B$ is simple, $\mathrm{K}_1(B)=0$, and $B$ is either projectionless or it is unital and its only non-zero projection is its unit then $\mathrm{sr}(A\otimes B)=1$ for every C*-algebra $A$ that is approximately contained in the class of RSH$_0$-algebras with 1-dimensional spectrum. 

\end{enumerate}
Moreover, if $B$ is a simple direct limit of a sequence of 1-dimensional NCCW-complexes with $\mathrm{K}_0(B)=\mathrm{K}_1(B)=0$, and $A$ is approximately contained in the class of C*-algebras that are stably isomorphic to a commutative C*-algebra, then $A\otimes B$ is approximately contained in the class of C*-algebras that are stably isomorphic to 1-dimensional NCCW-complexes. In particular the decomposition rank and the nuclear dimension of $A\otimes B$ is one. If in addition $A$ is separable then $A\otimes B$ can be written as an inductive limit of C*-algebras that are stably isomorphic to 1-dimensional NCCW-complexes.
\end{theorem}

Recursive subhomogeneous algebras (shortly RSH-algebras) were introduced and studied in \cite{PhillipsRSH} and \cite{PhillipsIRSH}. This class of C*-algebras arise naturally in the study of the crossed product C*-algebras obtained from a minimal homeomorphism of a compact metric space. In \cite{PhillipsIRSH} it is shown that if $B$ is a unital simple direct limit of a system of  RSH-algebras with no dimension growth then the stable rank of $B$ is one. This result was generalized in \cite{Hutian} for certain simple inductive limits of some stable version of RSH-algebras. In the first part of Theorem \ref{thm: mainmain} we prove a similar result for C*-algebras with a finite number of ideals that can be written as the direct limit of some non-unital version of RSH-algebras. The proof of this result follows closely the proof of the unital simple case. The second and last part of Theorem \ref{thm: mainmain} can be applied when $B$ belongs to the class of C*-algebras classified in \cite{Razak}; in particular, for the C*-algebra $\mathcal W$. These results are quite surprising since they imply that the tensor product of $\mathcal{W}$ with an arbitrary commutative C*-algebra can be decomposed as an inductive limit of 1-dimensional NCCW-complexes; and thus, it is of stable rank one. In other words, this result state that the stable rank and the dimension of the spectrum of a C*-algebra are independent when there are no K-theoretical obstructions. The third part of Theorem \ref{thm: mainmain} can be applied to the Jiang-Su algebra. This result provide some evidence for the following statement: the operation of taking tensor products with the Jiang-Su algebra does not increase the stable rank.   

The paper consists of three sections. In Section 2 we study pullbacks of C*-algebras; specifically, the structure of the ideals, quotients, and unitizations of pullbacks. In this section we also introduce and study the class of RSH$_0$-algebras. In Section 3 we prove the results stated in Theorem \ref{thm: mainmain}.

I am grateful to George Elliott, Huaxin Lin, Christopher Phillips, and Leonel Robert for interesting discussions concerning the subject matter of this paper.

\section{Preliminary definitions and results}

\subsection{Continuous lattices}
Let us briefly recall the definition of a continuous lattice (\cite[Definition I-1.6.]{Continuous-Domains}). Let $L$ be a lattice and let $a$ and $b$ be elements of $L$, we say that $a$ is {\it compactly contained} in $b$, denoted by $a\ll b$, if for every increasing net $(b_i)_{i\in I}$ with $b\le \sup_i b_i$, there exists $j$ such that $a\le b_j$. We say that $L$ is a {\it continuous lattice} if $L$ is complete (i.e., every subset of $L$ has a least upper bound or supremum), and if $a=\sup_{a'\ll a}a'$ for every $a\in L$. An element $a\in A$ is said to be {\it compact} if $a\ll a$.

Let $A$ be a C*-algebra and let $\mathrm{Lat}(A)$ denote the lattice of closed two-sided ideals of $A$. It was shown in \cite[Proposition I-1.21.2.]{Continuous-Domains} that $\mathrm{Lat}(A)$ is a continuous lattice. Given a subset $S$ of $A$ we denote by $\I(S)$ the closed two-sided ideal generated by $S$. If  $\phi\colon A\to B$ is a *-homomorphism of C*-algebras then the map $\check\phi\colon \Lat(A)\to \Lat(B)$ defined by $\check \phi(I)=\I(\phi(I))$ is a morphism in the category of continuous lattice. That is, is a lattice map that preserves the compact containment relation and suprema of arbitrary subsets.  

Throughout this paper an ideal of a C*-algebra will mean a closed two-sided ideal except when we refer to the Pedersen ideal (i.e., the smallest dense two sided-ideal of a C*-algebra (\cite[Theorem 5.6.1]{Pedersen})). The Pedersen ideal of a C*-algebra $A$ will be denoted by $\mathrm{Ped}(A)$.

The following proposition characterize the relation of compact containment of ideals. This is a refinement of 
\cite[Proposition I-1.21.1]{Continuous-Domains}.

\begin{proposition}\label{prop: far below}
Let $A$ be a C*-algebra and let $I$ and $J$ be closed two-sided ideals of $A$. Then $I\ll J$ if and only if there exist $a\in J_+$ and $\epsilon>0$ such that $I\subseteq \I((a-\epsilon)_+)$.
\end{proposition}
\begin{proof}
It was shown in \cite[Proposition I-1.21.1]{Continuous-Domains} that $I\subseteq \I((a-\epsilon)_+)$ implies $I\ll J$. Let us prove the opposite implication. By \cite[Proposition I-1.21.1]{Continuous-Domains} it is sufficient to show that for every $a_1, a_2, \cdots, a_n\in \Ped(A)_+$ there exist $a\in A$ and $\epsilon>0$ such that 
\[
\I(a_1, a_2, \cdots, a_n)\subseteq \I((a-\epsilon)_+).
\]
Suppose that $a_i=(b_i-\delta)_+$, where $b_1, b_2\cdots, b_n\in A_+$ and $\delta>0$. Set $a=\sum_{j=1}^n b_j$. Take an approximate unit $(u_i)_{i=1}^\infty$ of the hereditary algebra $\Her(a)$ such that $u_iau_i\le \left(a-1/i\right)_+$ 
(e.g., $u_i=h_i(a)$ with $h_i(t)=\frac{1}{t}(t-1/i)_+$). Then $u_ib_ju_i\to b_j$ since $b_1, b_2, \cdots, b_n\in \Her(a)$. By \cite[Lemma 2.2]{Kirchberg-Rordam} there exist $i\ge 1$ and $d_1, d_2, \cdots, d_n\in A$ such that 
\begin{align*}
(b_j-\delta)_+=d_j^*u_ib_ju_id_j,
\end{align*}
for all $1\le j\le n$.
It follows that
\begin{align*}
\I((b_1-\delta)_+, (b_2-\delta)_+, \cdots, (b_n-\delta)_+ )&\subseteq \I(u_ib_1u_i, u_ib_2u_i, \cdots, u_ib_nu_i)\\
&\subseteq \I(\sum_{j=1}^nu_ib_ju_i)\\
&\subseteq \I( (a-1/i)_+).
\end{align*}

Now let us prove the general case. By the construction of the Pedersen ideal of a C*-algebra (see the proof of \cite[Theorem 5.6.1]{PedersenBook}) there are $c_1, c_2, \cdots, c_n\in A_+$ and continuous functions $f_1, f_2, \cdots, f_n\colon (0,\infty)\to [0,\infty)$ with compact support such that
\[
a_1+a_2+\cdots+a_n\le f_1(c_1)+f_2(c_2)+\cdots +f_n(c_n).
\]
Choose positive elements $(b_i)_{i=1}^n\in A$ and a real number $\delta>0$ such that $f_i(c_i)=(b_i-\delta)_+$. It follows that
\begin{align*}
\I(a_1, a_2, \cdots, a_n)&\subseteq \I(a_1+a_2+\cdots+a_n)\\
&\subseteq \I(f_1(c_1), f_2(c_2), \cdots, f_n(c_n))\\
&\subseteq \I((b_1-\delta)_+, (b_2-\delta)_+, \cdots, (b_n-\delta)_+).
\end{align*}
Set $a=\sum_{j=1}^nb_j$. Then by the previous argument there exists $\epsilon>0$ such that
$$\I((b_1-\delta)_+, (b_2-\delta)_+, \cdots, (b_n-\delta)_+)\subseteq \I((a-\epsilon)_+).$$
\end{proof}

\begin{corollary}\label{cor: compact}
Let $A$ be a C*-algebra. The following are equivalent:
\begin{enumerate}
\item $I$ is a compact ideal of $A$ (i.e., $I\ll I$);
\item There is $a\in I_+$ and $\epsilon>0$ such that $I=\I((a-\epsilon)_+)$;
\item The spectrum of $I$ is compact.
\end{enumerate} 
\end{corollary}
\begin{proof}
The equivalence of (i) and (ii) follows by the previous proposition. The equivalence of (i) and (iii) follows using the standard identification of the closed two-sided ideals of a C*-algebra with the open subsets of its spectrum.
\end{proof}

Let $\phi\colon A\to B$ be a *-homomorphism of C*-algebras. We denote by $\check\phi\colon \Lat(A)\to \Lat(B)$ the lattice map defined by $\check \phi(I)=\I(\phi(I))$. The following proposition was proved in \cite[Proposition 3.1.9 (ii)]{LeonelPhD}:
\begin{proposition}\label{lem: idealIL}
Let $A=\varinjlim (A_n, \phi_{n, m})$ be an inductive limit of  C*-algebras. Let $I$, $J$, and $K$ be ideals of $A_k$ such that $I\ll J$ and 
$$\check\phi_{k,\infty}(J)\le \check\phi_{k,\infty}(K).$$ 
Then there is $m\ge k$ such that 
$$\check\phi_{k,m}(I)\ll \check\phi_{k,m}(K).$$
\end{proposition}

\begin{lemma}\label{idealsIL}
Let $A=\varinjlim (A_n, \phi_{n, m})$ be an inductive limit of  C*-algebras. Suppose that $A$ has compact spectrum. Then there exist $k\ge 1$ and ideals $I_n\subseteq A_n$, $n\ge k$, with compact spectrum such that $\check\phi_{n, n+1}(I_n)=I_{n+1}$ and $A=\varinjlim (I_n, \phi_{n, m}|_{I_n})$.
\end{lemma}
\begin{proof}
Let $A$ be as in the statement of the lemma. Then $A\ll A$ by Corollary \ref{cor: compact}. Also, we have
$$A=\overline{\bigcup_n\I(\phi_{n,\infty}(A_n))}=\sup_n \check\phi_{n,\infty}(A_n).$$
By the definition of the compact containment relation there is $m$ such that $\check\phi_{m,\infty}(A_m)=A$. Write $A_m=\sup_{I\ll A_m}I$. Then
$$A=\check\phi_{m,\infty}(A_m)=\sup_{I\ll A_m}\check\phi_{m,\infty}(I).$$
Since $A\ll A$ there are ideals $J, K\subseteq A_m$ with $J\ll K\ll A_m$ such that 
$$\check\phi_{m,\infty}(J)=\check\phi_{m,\infty}(K)=A.$$ 
This implies, by Proposition \ref{lem: idealIL} applied to $I'=K$, $J'=A_n$, and $K'=J$, that $\check\phi_{m,k}(K)\ll\check\phi_{m,k}(J)$ for some $k\ge m$. Hence, 
$$\check\phi_{m,k}(J)\ll \check\phi_{m,k}(K)\ll \check\phi_{m,k}(J).$$
In other words, the spectrum of $\check\phi_{m,k}(J)$ is compact. The statement of the lemma now follows by taking $I_n=\check\phi_{m,n}(J)$. 
\end{proof}

\begin{lemma}\label{lem: idealgen}
Let $A=\varinjlim (A_n,\phi_{n,m})$ be an inductive limit of C*-algebras such that $\check\phi_{n,n+1}(A_n)=A_{n+1}$ for all $n$. Let $a\in A_m$, for some $m$, be such that $\I(\phi_{m,\infty}(a))=A$ and suppose that $A_m$ has compact spectrum. Then there exists $l\ge m$ such that $\I(\phi_{m,k}(a))=A_k$ for all $k\ge l$. 
\end{lemma}
\begin{proof}
Let $a\in A_m$ be as in the statement of the lemma. Since $A_m$ has compact spectrum, $A_m\ll A_m$ by Corollary \ref{cor: compact}. Therefore,
\begin{align*}
&A=\check\phi_{m,\infty}(A_m)\ll \check\phi_{m,\infty}(A_m)=A,\\
&A_n=\check\phi_{m,n}(A_m)\ll \check\phi_{m,n}(A_m)=A_n,
\end{align*}
for all $n\ge m$. In other words, $A$ and $A_n$, with $n\ge m$, have compact spectrum.
Using the definition of the compact containment relation and that $A\ll A=\sup_n\check\phi_{n,\infty}(A_n)$ we get
\begin{align*}
\check\phi_{n,\infty}(A_n)=A=\I(\phi_{m,\infty}(a)),
\end{align*}
for all $n$. This implies, by Proposition \ref{lem: idealIL} applied to $I=A_m$, $J=A_m$, and $K=\I(a)$, that there exists $l\ge m$ such that 
$$\check\phi_{m,k}(A_m)\subseteq\I(\phi_{m,k}(a)),$$
for all $k\ge l$.
Since, we also have 
$$\I(\phi_{m,k}(a))\subseteq \check\phi_{m,k}(A_m),\quad \check\phi_{m, k}(A_m)=A_k,$$
we get $\I(\phi_{m,k}(a))=A_k$ for all $k\ge l$.
\end{proof}

\subsection{Ideals, quotients, and unitizations of pullbacks}
Recall that if $A$, $B$, and $C$ are C*-algebras and $\phi\colon A\to C$ and $\psi\colon B\to C$ are *-homomorphisms then the pullback $A\oplus_C B$ is given by
\begin{align*}
A\oplus_C B=\{(a,b)\in A\oplus B: \phi(a)=\psi(b)\}.
\end{align*}
We denote by $\pi_1\colon A\oplus_C B\to A$ and $\pi_2\colon A\oplus_C B\to B$ the projections maps associated to this pullback. That is, the maps defined by $\pi_1(a,b)=a$ and $\pi_2(a,b)=b$ for $(a,b)\in A\oplus_C B$.

The following is Proposition 3.1. of \cite{Pedersen}.

\begin{proposition}\label{prop: Pedersen}
A commutative diagram of C*-algebras 
\begin{equation*}
\xymatrix{X \ar[d]_{\pi_1} \ar[r]^{\pi_2} & B\ar[d]^\psi \\ A\ar[r]_\phi & C}
\end{equation*}
is a pullback if and only if the following conditions hold:
\begin{itemize}
\item[(i)] $\mathrm{Ker}(\pi_1)\cap \mathrm{Ker}(\pi_2)=\{0\}$,
\item[(ii)] $\psi^{-1}(\phi(A))=\pi_2(X)$,
\item[(iii)] $\pi_1(\mathrm{Ker}(\pi_2))=\mathrm{Ker}(\phi)$.
\end{itemize}
\end{proposition}

\subsubsection{Ideals of Pullbacks}
The following theorem describes the lattice of ideals of a pullback of C*-algebras.


\begin{theorem}\label{th: ideals}
Let $A$, $B$, and $C$ be C*-algebras and let $\phi\colon A\to C$ and $\psi\colon B\to C$ be *-homomorphisms with $\phi$ is surjective. Let $\mathrm{L}_{A, B}$ be the subset of $\Lat(A)\times \Lat(B)$ consisting of pairs of ideals $(I, J)$ such that:
\begin{enumerate}
\item[(1)] $\phi(I)=\check\psi(J)$,
\item[(2)] $\I(I\cap \phi^{-1}(\psi(J)))=I$.
\end{enumerate}
If $\mathrm{L}_{A, B}$ is endowed with the order 
\[
(I_1, J_1)\le (I_2, J_2)\quad \text{ if }\quad I_1\le I_2,  J_1\le J_2, 
\]
then:
\begin{itemize}
\item[(i)] $\mathrm{L}_{A, B}$ is a continuous lattice and the map $\alpha\colon \Lat(A\oplus_CB)\to \mathrm{L}_{A, B}$ defined by
\begin{align*}
\alpha(M)=(\check\pi_1(M), \check\pi_2(M)),
\end{align*}
is an isomorphism of  lattices. Moreover, the inverse of $\alpha$ is given by the map $\beta\colon \mathrm{L}_{A, B}\to\Lat(A\oplus_CB)$ defined by
\begin{align*}
\beta(I, J)=I\oplus_{\phi(I)} J,
\end{align*}
where $I\oplus_{\phi(I)} J$ denotes the pullback taken with respect to the restriction maps $\phi|_I\colon I\to \phi(I)$ and $\psi|_J\colon J\to \phi(I)$ (this pullback may be identified with the subset of $A\oplus_C B$ given by $\{(a,b)\in A\oplus_CB\mid a\in I, b\in J\}$);
\item[(ii)] For every subset $T$ of $\mathrm{L}_{A, B}$ the supremum of $T$ in $\mathrm{L}_{A, B}$ is equal to the supremum of $T$ in $\Lat(A)\times\Lat(B)$;
\item[(iii)] $(I_1, J_1)\ll (I_2, J_2)$ in $\mathrm{L}_{A, B}$ if and only if $(I_1, J_1)\ll (I_2, J_2)$ in $\Lat(A)\times\Lat(B)$.
\end{itemize} 
\end{theorem}
\begin{proof}
(i) Let $M$ be an ideal of $A\oplus_CB$. Then since $\phi$ is surjective $\pi_2$ is surjective by (ii) of Proposition \ref{prop: Pedersen}. Hence, $\pi_2(M)$ is an ideal of $B$ and $\check\pi_2(M)=\pi_2(M)$. Let $I=\check\pi_1(M)$ and $J=\pi_2(M)$. Then it is clear that $$\phi(I)=\I(\psi(J))=\check\psi(J).$$ 
Also, since $\pi_1(M)\subseteq I$ and $\pi_1(M)\subseteq\phi^{-1}(\psi(J))$ we have $I=\I(I\cap \phi^{-1}(\psi(J)))$.
It follows that $(I, J)\in \mathrm{L}_{A, B}$ and that the map $\alpha$ is well defined. It is straightforward to check using Proposition \ref{prop: Pedersen} that the diagram
\begin{equation*}
\xymatrix{M \ar[d]_{\pi_1|_M} \ar[r]^{\pi_2|_M} & J\ar[d]^{\psi|_J} \\ I\ar[r]_{\phi|_I} & \phi(I)}
\end{equation*}
is a pullback. Therefore
$M=I\oplus_{\phi(I)} J$ or what is the same $(\beta\circ \alpha)(M)=M$. Since $M$ is an arbitrary ideal of $A\oplus_CB$ we conclude that $\beta\circ \alpha=\mathrm{id}_{\Lat(A\oplus_CB)}$.

Now let us show that $\alpha\circ\beta=\mathrm{id}_{\mathrm{L}_{A, B}}$.
Let $(I, J)\in \mathrm{L}_{A, B}$. Then $\pi_2(I\oplus_KJ)=J$ by (ii) of Proposition \ref{prop: Pedersen}.
Let $a$ be an element of $I\cap \phi^{-1}(\psi(J))$. Choose $b\in J$ such that $\phi(a)=\psi(b)$. Then $(a,b)\in I\oplus_{\phi(I)}J$ and $a\in \pi_1(I\oplus_{\phi(I)}J)$. Since $a$ is arbitrary this implies that 
$$I\cap \phi^{-1}(\psi(J))\subseteq \pi_1(I\oplus_{\phi(I)}J)\subseteq I.$$ 
Therefore $\check\pi_1(I\oplus_{\phi(I)}J)=I$ by (2) of the definition of $\mathrm{L}_{A, B}$. We now have
\begin{align*}
(\alpha\circ\beta)(I,J)=\alpha(I\oplus_{\phi(I)}J)=(\check\pi_1(I\oplus_{\phi(I)}J), \check\pi_2(I\oplus_{\phi(I)}J))=(I, J).
\end{align*}
Therefore, $\alpha\circ\beta=\id_{\mathrm{L}_{A, B}}$.

The maps $\alpha$ and $\beta$ are clearly order-preserving. Hence, they are isomorphisms of lattices that are inverse of each other. Since $\Lat(A\oplus_CB)$ is a continuous lattice it follows that $\mathrm{L}_{A, B}$ is a continuous lattice. 

(ii) Let $T$ be a subset of $\mathrm{L}_{A, B}$ and let 
\begin{align*}
\tilde{I}=\I(\{I\mid (I, J)\in T\}),\quad
\tilde{J}=\I(\{J\mid (I, J)\in T\}).
\end{align*}
Then
\[
(\tilde I, \tilde J)=\I(\{(I, J) \mid (I, J)\in T\})=\sup \{(I, J) \mid (I, J)\in T\},
\]  
where the supremum is taken in $\Lat(A)\times\Lat(B)$. Let us see that $(\tilde I, \tilde J)\in \mathrm{L}_{A, B}$. We have
\begin{align*}
\phi(\tilde I)&=\I(\{\phi(I) \mid (I,J)\in T\})\\
&=\I(\{\check\psi(J) \mid (I,J )\in T\})\\
&=\I(\{\I(\psi(J)) \mid (I,J )\in T\})\\
&=\I(\{\psi(J) \mid (I, J)\in T\})\\
&=\check\psi(\tilde J).
\end{align*}
Hence, $(\tilde I, \tilde J)$ satisfies condition (1) of the definition of $\mathrm{L}_{A,B}$.

Let $(I, J)\in T$ and let $G=I\cap \phi^{-1}(\psi(J))$. Then $\I(G)=I$, $G\subseteq \tilde I$, and 
\[
G\subseteq \phi^{-1}(\psi(J))\subseteq \phi^{-1}(\psi(\tilde J)). 
\]
Hence, $G\subseteq \tilde I\cap \phi^{-1}(\psi(\tilde J))$. Using that $\I(G)=I$ we get
\[
I\subseteq \I(\tilde I \cap \phi^{-1}(\psi(\tilde J))).
\]
Therefore,
\[
\tilde I=\I(\{I \mid (I, J)\in T\})\subseteq \I(\tilde I \cap \phi^{-1}(\psi(\tilde J))).
\]
Since $\I(\tilde I \cap \phi^{-1}(\psi(\tilde J)))\subseteq \tilde I$ we conclude that $\tilde I=\I(\tilde I\cap \phi^{-1}(\psi(\tilde J)))$. This shows that  $(\tilde I, \tilde J)$ satisfies condition (2) of the definition of $\mathrm{L}_{A,B}$.
Therefore, $(\tilde I, \tilde J)\in \mathrm{L}_{A,B}$. Since the order in $\mathrm{L}_{A, B}$ is induced by the order in $\Lat(A)\times \Lat(B)$, $(\tilde I, \tilde J)$ is the supremum of the set $T$ in $\mathrm{L}_{A, B}$.

(iii)  Suppose that $(I_1, J_1)\ll (I_2, J_2)$ in $L_{A,B}$. Then 
$$I_1\oplus_{\phi(I_1)}J_1=\beta(I_1, J_1)\ll \beta (I_2, J_2)=I_2\oplus_{\phi(I_2)}J_2.$$
By Proposition \ref{prop: far below} there are $(a,b)\in I_2\oplus_{\phi(I_2)}J_2$ and $\epsilon>0$ such that
\[
I_1\oplus_{\phi(I_1)} J_1\subseteq \I((a,b)-\epsilon)_+)=\I((a-\epsilon)_+, (b-\epsilon)_+).
\]
This implies that
\[
(I_1, J_1)=\alpha(I_1\oplus_{\phi(I_1)} J_1)\subseteq (\I((a-\epsilon)_+), \I((b-\epsilon)_+)).
\]
Therefore, $I_1\ll I_2$ and $J_1\ll J_2$ by Proposition \ref{prop: far below}. In other words, $(I_1, J_1)\ll (I_2, J_2)$ in $\Lat(A)\times \Lat(B)$.

Let $(I_1, J_1), (I_2, J_2)\in \Lat(A)\times \Lat(B)$ be such that $(I_1, J_1)\ll (I_2, J_2)$ in $\Lat(A)\times \Lat(B)$. Let $((I_\lambda, J_\lambda))_{\lambda\in \Lambda}\subseteq \mathrm{L}_{A, B}$ be an increasing net such that $(I_2, J_2)= \sup_{\lambda\in \Lambda} (I_\lambda, J_\lambda)$, where the supremum is taken in $\mathrm{L}_{A, B}$. Then by the second part of the theorem the supremum of $((I_\lambda, J_\lambda))_{\lambda\in \Lambda}$ in $\Lat(A)\times \Lat(B)$ is equal to $(I_2, J_2)$. Since $(I_1, J_1)\ll (I_2, J_2)$ in $\Lat(A)\times \Lat(B)$ there exists $\lambda\in \Lambda$ such that $(I_1, J_1)\le (I_\lambda, J_\lambda)$
in $\Lat(A)\times\Lat(B)$. This implies that $(I_1, J_1)\le (I_\lambda, J_\lambda)$ in $\mathrm{L}_{A, B}$. Since the increasing net $((I_\lambda, J_\lambda))_{\lambda\in \Lambda}$ is arbitrary, we conclude by the definition of the compact containment relation that $(I_1, J_1)\ll (I_2, J_2)$ in $\mathrm{L}_{A, B}$. 


\end{proof}


\begin{corollary}\label{cor: compact ideals}
Let $A$, $B$, and $C$ be C*-algebras and let $\phi\colon A\to C$ and $\psi\colon B\to C$ be *-homomorphisms with $\phi$ is surjective. If $M$ be an ideal of $A\oplus_B C$ then there are ideals $I\subseteq A$, $J\subseteq B$, and $K\subseteq C$ such that 
$$M=I\oplus_K J, \quad \phi(I)=K, \quad\I(\psi(J))=K.$$ 
Moreover, if $M$ has compact spectrum then the ideals $I$, $J$, and $K$ may be taken with compact spectrum.
\end{corollary}

\subsubsection{Quotients of pullbacks}

Let $I\subseteq A$, $J\subseteq B$, and $K\subseteq C$ be ideals such that $\phi(I)=K$ and $\psi(J)\subseteq K$. Then the maps $\phi$ and $\psi$ induce *-homomorphisms $\overline\phi\colon A/I\to C/K$ and $\overline\psi\colon B/J\to C/K$ between the quotient C*-algebras. Therefore, we can form the pullback $(A/I)\oplus_{C/K}(B/J)$. In the following proposition we show that this pullback is natural isomorphic to the quotient of $A\oplus_CB$ by the ideal $I\oplus_KJ$.

\begin{proposition}\label{prop: quotients}
Let $A$, $B$, and $C$ be C*-algebras and let $\phi\colon A\to C$ and $\psi\colon B\to C$ be *-homomorphisms with $\phi$ is surjective. If $I\subseteq A$, $J\subseteq B$, and $K\subseteq C$ are ideals such that $\phi(I)=K$ and $\psi(J)\subseteq K$ then the map
\[
\gamma\colon (A\oplus_CB)/(I\oplus_KJ)\to (A/I)\oplus_{C/K}(B/J),
\] 
given by
\[
\gamma((a, b)+I\oplus_KJ)=(a+I, b+J),
\]
is a *-isomorphism.
\end{proposition}
\begin{proof}
Let $\rho\colon A\oplus_CB\to (A/I)\oplus_{C/K}(B/J)$ be the map defined by $\rho(a,b)=(a+I, b+J)$ for every $(a,b)\in A\oplus_CB$. Note that $\rho$ is well defined and that the kernel of $\rho$ is $I\oplus_K J$. 

Let us show that $\rho$ is surjective.
Let $(a+I,b+J)\in (A/I)\oplus_{C/K}(B/J)$. Then $\overline\phi(a+I)=\overline\psi(b+J)$ and so $\phi(a)-\psi(b)\in K$. Since $\phi$ is surjective there exists $a'\in A$ such that $\phi(a')=\psi(b)$. It follows that $\phi(a-a')\in K$. Since by assumption $\phi(I)=K$ there exists $a''\in I$ such that $\phi(a'')=\phi(a-a')$. We now have $\phi(a-a'')=\phi(a')=\phi(b)$. Hence, $(a-a'',b)\in A\oplus_CB$ and $\rho(a-a'',b)=(a+I, b+J)$. In other words, $\rho$ is surjective.

Since $\rho$ is surjective and the kernel of $\rho$ is $I\oplus_KJ$, $\rho$ induces an isomorphism $\overline\rho\colon (A\oplus_CB)/(I\oplus_KJ)\to (A/I)\oplus_{C/K}(B/J)$. It is trivial to see that $\overline\rho=\gamma$.
\end{proof}

\subsubsection{Unitization of Pullbacks}
Let $A$ be a C*-algebra. Let us denote by $\widetilde{A}$ the C*-algebra obtained by adjoining a unit to $A$ (note that if $A$ is unital $\widetilde{A}\cong A\oplus \C$). If $\phi\colon A\to B$ is a *-homomorphism then it extends to a unique *-homomorphism $\widetilde \phi\colon \widetilde A\to \widetilde B$ such that $\widetilde \phi(1_{\widetilde A})=1_{\widetilde B}$. If $B$ is unital then we denote by $\phi^\dagger\colon \widetilde A\to B$ the *-homomorphism that agrees with $\phi$ on $A$ and that satisfies $\phi^\dagger(1_{\widetilde A})=1_{B}$.

The following proposition can be easily verified using Proposition \ref{prop: Pedersen}.

\begin{proposition}\label{prop: unitization}
Let $A$, $B$ and $C$ be C*-algebras and let $\phi\colon A\to C$ and $\psi\colon B\to C$ be *-homomorphisms with $\phi$ is surjective. The following statements hold:
\begin{enumerate}
\item $(A\oplus_C B)\widetilde{\,\,}\cong \widetilde A\oplus_{\widetilde C} \widetilde B$. Specifically, the diagram
\begin{equation*}
\xymatrix{(A\oplus_C B)\widetilde{\,\,} \ar[d]_{\widetilde \pi_1} \ar[r]^{\,\,\quad\widetilde \pi_2} & \widetilde B\ar[d]^{\widetilde \psi} \\ \widetilde A\ar[r]_{\widetilde \phi} & \widetilde C}
\end{equation*}
is a pullback;
\item If $A$ is unital then $(A\oplus_C B)\widetilde{\,\,}\cong A\oplus_{C} \widetilde B$. Specifically, the diagram
\begin{equation*}
\xymatrix{(A\oplus_C B)\widetilde{\,\,} \ar[d]_{\pi_1^\dagger} \ar[r]^{\,\,\quad\widetilde \pi_2} & \widetilde B\ar[d]^{\psi^\dagger} \\ A\ar[r]_{\phi} & C}
\end{equation*}
is a pullback.
\end{enumerate}
\end{proposition}

\subsection{Recursive subhomogeneous algebras} 
In this subsection we introduce a nonunital version of recursive subhomogeneous algebras (cf., \cite{PhillipsRSH}).

\begin{definition}\label{def: RSH0}
Let $A$ be a C*-algebra. We say that $A$ is a recursive subhomogeneous algebra vanishing at infinity, shortly RSH$_0$-algebra, if $A$ is isomorphic to the $k$-th term, for some $k$, of a sequence of iterated pullbacks of the form:
\begin{align*}
P_i=
\begin{cases}
\mathrm{C}_0(X_0, \M_{n_0}), \quad \text{if}\quad i=0,\\
\mathrm{C}_0(X_{i}, \M_{n_{i}})\oplus_{\mathrm{C}_0\left(X_{i}^{(0)}, \, \M_{n_{i}}\right)} P_{i-1}, \quad\text{if}\quad 1\le i\le k,
\end{cases}
\end{align*}
with $X_i$ locally compact and Hausdorff, $X_i^{(0)}$ a closed subspace of $X_i$, $n_i$ a positive integer, and where the maps $\mathrm{C}_0(X_i, \M_{n_i})\to \mathrm{C}_0(X_i^{(0)}, \M_{n_i})$ are always restriction maps. We call $(P_{i}, X_{i+1}, X_{i+1}^{(0)}, n_{i+1} ,\rho_i)_{i=0}^{k-1}$ a decomposition of $A$. Here $\rho_i\colon P_i\to \mathrm{C}_0(X_{i+1}^{(0)},\M_{n_{i+1}})$ denotes the *-homomorphism appearing in the $i$-th pullback. Associated to this decomposition are:
\begin{enumerate}
\item its length $k$;
\item its base spaces $X_0, X_1, \cdots, X_k$ and total space $X=\coprod_{i=0}^kX_i$;
\item its topological dimension $\dim X$ (i.e., the covering dimension of $X$), and topological dimension function $d\colon X\to \N\cup \{0\}$, defined by $d(x)=\dim (X_k)$ when $x\in X_k$;
\item its matrix sizes $n_0, n_1, \cdots, n_k$, and matrix size function $m\colon X\to \N\cup \{0\}$, defined by $m(x)=n_k$ when $x\in X_k$;
\item its standard representation $\sigma\colon A\to \bigoplus_{i=0}^k\mathrm{C_0}(X_i, \M_{n_i})$, defined by using the natural inclusion of a pullback on the direct sum of its projection C*-algebras;
\item the evaluation maps $ev_x\colon A\to \M_{n_i}$ for $x\in X_i$, defined as the composition of $\sigma$ with the evaluation map at $x$.
\end{enumerate}

\end{definition}

Note that by definition (see \cite[Definition 1.2]{PhillipsRSH}) a C*-algebra $A$ is a recursive subhomogeneous algebra, shortly a RSH-algebra, if it admits a RSH$_0$-decomposition in which all the algebras and all the *-homomorphisms are unital.



\begin{proposition}\label{prop: permanence}
Let $A$ be a RSH$_0$-algebra. The following statements hold:
\begin{enumerate}
\item If $I$ is an ideal of $A$ then $I$ and $A/I$ are RSH$_0$-algebras;

\item $A$ has compact spectrum if and only if $A$ has a decomposition
$$(P_{i}, X_{i+1}, X_{i+1}^{(0)}, n_{i+1}, \rho_i)_{i=0}^{k-1},\quad \rho_i\colon P_i\to \mathrm{C}_0(X_{i+1}^{(0)}, \M_{n_{i+1}}), $$
such that
\begin{enumerate}
\item $X_i$ and $X_i^{(0)}$ are compact for all $i$;
\item $\rho_i\colon P_i \to \mathrm{C}(X_{i+1}^{(0)},\M_{n_{i+1}})$ satisfies
$$\I(\rho_i(P_i))=\mathrm{C}_0\left(X_{i+1}^{(0)},\M_{n_{i+1}}\right),$$
for all $i$;

\item there is $l$, with $0\le l\le k$, such that $P_l$ is unital, $P_{l+1}$ is non-unital (if $l<k$), and $(P_{i}, X_{i+1}, X_{i+1}^{(0)}, n_{i+1}, \rho_i)_{i=0}^{l-1}$ is a RSH-decomposition of $P_l$ (i.e., all the maps and all the algebras in the decomposition are unital);
\end{enumerate}

\item If $A$ has compact spectrum and 
$$(P_{i}, X_{i+1}, X_{i+1}^{(0)}, n_{i+1},\rho_i)_{i=0}^{k-1},\quad\rho_i\colon P_i\to \mathrm{C}_0(X_{i+1}^{(0)}, \M_{n_{i+1}}),$$ 
is a decomposition of $A$ satisfying (a), (b), and (c) of (ii) then $\widetilde A$ is RSH-algebra with decomposition $(Q_i, Y_{i+1}, Y_{i+1}^{(0)}, m_{i+1}, \mu_i)_{i=0}^{k}$ given by
\begin{align*}
(Q_i, Y_{i+1}, Y_{i+1}^{(0)}, m_{i+1}, \mu_i)=
\begin{cases}
(P_i, X_{i+1}, X_{i+1}^{(0)}, n_{i+1}, \rho_i) &\text{if }\, 0\le i\le l-1,\\
(P_l, \{y_l\}, \varnothing, 1, \mu_l\colon P_l\to 0) &\text{if }\, i=l,\\  
(\widetilde P_{i-1}, X_i, X_i^{(0)}, n_{i+1}, \rho_{i-1}^\dagger) &\text{if }\, l+1\le i\le k. 
\end{cases}
\end{align*}
(In particular, $Q_{l+1}\cong Q_l\oplus \C$.)
\item Let $X$ be a locally compact Hausdorff space then $\mathrm{C}_0(X)\otimes A$ is a RSH$_0$-algebra. Moreover, if $X$ is compact and $A$ has compact spectrum then $\mathrm{C}(X)\otimes A$ has compact spectrum;

\item If $A$ has compact spectrum and $p\in A$ is a projection then $pAp$ is a RSH-algebra.
\end{enumerate}

In addition, if a RSH$_0$-decomposition of $A$ is given then the decompositions of $I$ and $A/I$ in (i),  $A$ in (ii), and $pAp$ in (v) can be taken with topological dimension no greater than the topological dimension of the given decomposition of $A$. 
\end{proposition}

\begin{proof}
The first part of the proposition is a consequence of Corollary \ref{cor: compact ideals} and Proposition \ref{prop: quotients}. The second part follows using Corollary \ref{cor: compact ideals} and that every ideal with compact spectrum of a C*-algebra of the form $\M_n(\mathrm C(X))$, with $X$ locally compact and Hausdorff, has the form $\M_n(\mathrm C(Y))$ for some compact open subset $Y$ of $X$. The third part follows by Proposition \ref{prop: unitization}; that is, using that the following diagrams are pullbacks:
\begin{equation*}
\xymatrix{\widetilde P_l \ar[d] \ar[r] &P_l \ar[d] \\ \C(\cong\mathrm{C}(\{y_l\})\ar[r] & 0}
\quad\qquad
\xymatrix{\widetilde P_{i+1} \ar[d]_{\left(\pi_1^{(i+1)}\right)^\dagger} \ar[r]^{\widetilde{\pi_2^{(i+1)}}} & \widetilde P_i\ar[d]^{\rho_i^\dagger} \\ \mathrm{C}(X_{i+1}, \M_{n_{i+1}})\ar[r] & \mathrm{C}(X_{i+1}^{(0)}, \M_{n_{i+1}})}
\end{equation*}
for $l\le i\le k-1$, where $\pi_1^{(i+1)}\colon P_{i+1}\to \mathrm{C}(X_{i+1}, \M_{n_{i+1}})$ and 
$\pi_2^{(i+1)}\colon P_{i+1}\to P_i$ denote the projection maps.
The fourth part of the proposition is a consequence of \cite[Theorem 3.9]{Pedersen}. 
Let us prove the last part of the proposition. By (i) and (ii) the C*-algebra $\widetilde A$ is a RSH-algebra. Hence, $pAp=p\widetilde A p$ is a RSH-algebra by \cite[Corollary 1.11]{PhillipsRSH}.
\end{proof}

The following theorem is a corollary of the previous proposition and \cite[Corollary 2.1]{Ping-Winter}:
\begin{proposition}\label{thm: ASH}
Let $A$ be a separable C*-algebra that can be written as the inductive limit of subhomogeneous algebras. Then $A$ can be written as the inductive limit of RSH$_0$-algebras of finite topological dimension.
\end{proposition}
\begin{proof}
Let $A$ be as in the statement of the theorem. Since the class of subhomogeneous C*-algebras is closed under the operation of adjoining a unit the C*-algebra $\widetilde A$ is the inductive limit of unital subhomogeneous C*-algebras. It follows now by \cite[Corollary 2.1]{Ping-Winter} that $\widetilde A$ is the inductive limit of RSH-algebras. Since $A$ is an ideal of $\widetilde A$ it can be written as an inductive limit of ideals of the finite stage C*-algebras. By (i) of Proposition \ref{prop: permanence} these ideals are RSH$_0$-algebras that can be taken to have topological dimension no greater than the topological dimension of the given algebras. In particular, they can be taken with finite topological dimension.
\end{proof}

\begin{definition}\label{def: standarddecomposition}
Let $A$ be a RSH$_0$-algebra with compact spectrum. We say that a decomposition of $A$ is a standard decomposition if it satisfies (a), (b), and (c) of (ii) of Proposition \ref{prop: permanence}. The integer $l$ associated to this decomposition as in (c) will be called the unital degree of the decomposition.
\end{definition}

Let $A$ be a C*-algebra. We denote by $\mathrm{U}(A)$ and $\mathrm{U}_0(A)$ the unitary group of the unitization $\widetilde A$ of $A$ and the subgroup of $\mathrm{U}(A)$ consisting of the unitaries that are connected to the identity. If $a\in A$ we denote by $|a|$ the element $(a^*a)^{\frac 1 2}$.

The following proposition is a minor modification of \cite[ Proposition 3.4]{PhillipsIRSH} (the assumption on the dimension of the subspaces $E_x$ is only assumed to hold on the total space of $A$ instead of the total space of $\widetilde A$). 

\begin{proposition}\label{prop: polar}
Let $A$ be a RSH$_0$-algebra with compact spectrum. Suppose that a standard decomposition of $A$ is given and let $X$, $m$, and $d$ be its total space, matrix size function, and topological dimension function (as in Definition \ref{def: RSH0}). Let $\alpha,\epsilon>0$. Let $a\in \widetilde A$, and suppose that for every $x\in X$ there is a subspace $E_x$ of $\C^{m(x)}$ with $\dim(E_x)\ge \frac 1 2 d(x)$ and such that $\|\mathrm{ev}_x(a)\xi\|<\alpha\|\xi\|$ for $\xi\in E_x\setminus\{0\}$. Then there is a unitary $u\in \mathrm{U}_0(A)$ such that 
\begin{align}\label{eq: polar}
\|u|a|-a\|<2\alpha+\epsilon.
\end{align}
\end{proposition}
\begin{proof}
Let $k$, $l$, and $X$ be the length, unital degree, and total space of the given standard decomposition of $A$. By (iii) of Proposition \ref{def: standarddecomposition} the total space of $\widetilde A$ is the disjoin union of $X$ and a singleton space $\{x_0\}$. Let us prove the following by induction in the number $k-l$: Let $\epsilon>0$. Let $a\in \widetilde A$, with $\|a\|\le 1$, be as in the statement of the proposition and let $p(x)=\chi_{(-\infty, \alpha)}(\mathrm{ev}_x(|a|))$ for $x\in X\sqcup\{x_0\}$. Then there is a unitary $u\in \mathrm{U}_0(A)$ such that
\begin{align*}
\|\mathrm{ev}_x(u|a|-a)(1-p(x))\|<\epsilon,
\end{align*}   
for all $x\in X\sqcup\{x_0\}$.
As in the proof of \cite[Proposition 3.4]{PhillipsIRSH} the previous inequality and the inequality $\|\mathrm{ev}_x(a)p(x)\|\le \alpha$, for $x\in X\sqcup \{x_0\}$, imply the inequality \eqref{eq: polar} (the former clearly holds by the definition of $p$).   

If $k-l=0$ then $A$ is unital and so it is a RSH-algebra. Also, $\widetilde A\cong A\oplus \C$, where $\C$ is identified with $\mathrm{C}(\{x_0\})$. Let us write $a$ as pair $(b, \lambda)\in A\oplus \C$. Then $b$ satisfies the conditions of \cite[Proposition 3.4]{PhillipsIRSH}. Therefore, by the result stated in proof of the same proposition (see the second paragraph of the proof) there is a unitary $v\in A$ such that  
\begin{align*}
\|\mathrm{ev}_x(v|b|-b)(1-p(x))\|<\epsilon,
\end{align*}   
for all $x\in X$. The case $k-l=0$ now follows by taking $u=(v, \lambda')$, where $\lambda'=1$ if $\lambda=0$ and $\lambda'=\frac{\lambda}{|\lambda|}$ otherwise.

Now suppose that the result holds for all standard decompositions with $k-l=N>0$ and suppose that for the given standard decomposition of $A$ one has $k-l=N+1$. Write $A=B\oplus_{\mathrm{C}(X^{(0)}, \M_n)}\mathrm{C}(X, \M_n)$, where $B$ denotes the $(k-1)$-pullback in the given decomposition. Then the difference between the length of $B$ and its unital degree is $N$. 
Also, if $Y$ denote the total space of $B$ then the total spaces of $\widetilde B$, $A$, and $\widetilde A$ are $Y\sqcup \{x_0\}$, $X\sqcup Y$, and $X\sqcup Y\sqcup \{x_0\}$, respectively. Let $b$ be the image of $a$ in $\widetilde B$ and let $a_0$ be the image of $a$ in $\mathrm{C}(X, \M_n)$ under the projections maps given by the pullback. Then $\chi_{(-\infty, \alpha)}(\mathrm{ev}_y(|b|))=p(y)$ for $y\in Y\sqcup \{x_0\}$. For the given values of $\alpha$, $\epsilon$, and $n$, let $\delta$ be as in \cite[Lemma 3.3]{PhillipsIRSH}. Since $b\in \widetilde B$ satisfies the conditions of the proposition and $\widetilde B$ satisfies the induction assumptions there is a unitary $v\in \mathrm{U}_0(B)$ such that 
$$\|\mathrm{ev}_y(v|b|-b)(1-p(y))\|<\delta$$
for all $y\in Y$.
The rest of proof now follows line by line the proof of \cite[Proposition 3.4]{PhillipsIRSH} starting in the fourth paragraph.
\end{proof}

\begin{definition}\label{def: dimension growth}
Let $(A_i, \phi_{i,j})_{i,j}$ be a system of RSH$_0$-algebras. For each $i$ let $X_i$ denote the total space of $A_i$ ((ii) of Definition \ref{def: RSH0}). We say that the system $(A_i, \phi_{i,j})_{i,j}$ has no dimension growth if $\sup_i(\dim X_i)<\infty$.
\end{definition}

\begin{lemma}\label{lem: rank}
Let $A=\varinjlim (A_i, \phi_{i,j})$ be an inductive limit of a system of RSH$_0$-algebras with no dimension growth. Suppose that for each $i\ge 1$ the C*-algebra $A_i$ has compact spectrum and that the given decomposition of $A_i$ is standard. Let $X_i$ be the total space of $A_i$. Let $a\in \widetilde{A}_k$, for some $k$, be a noninvertible  element. Let $\epsilon>0$. Suppose that for every $N\in \N$ there exists $j\ge k$ such that for every $x\in X_j$,  
\begin{align}\label{eq: rank}
\mathrm{rank}\left(\mathrm{ev}_x(\widetilde{\phi}_{k,j}(f_\epsilon(|a|)))\right)\ge N,
\end{align}
for some continuous function $f_\epsilon\colon (0,\infty)\to [0,1]$ satisfying $f_\epsilon(x)>0$ for $x\in (0,\epsilon)$ and $f_\epsilon(x)=0$ for $x\in [\epsilon , \infty)$. Then there exists a unitary $u\in \widetilde{A}$ such that
\begin{align*}
\|\widetilde{\phi}_{k,\infty}(a)-u|\widetilde{\phi}_{k,\infty}(a)|\|<5\epsilon.
\end{align*}
\end{lemma}
\begin{proof}
Let $m_i\colon X_i\to \N$ be the matrix size function of $A_i$. Let $N$ be a number such that $N>\sup_i\dim X_i$ ($N$ is finite since by assumption the system $(A_i, \phi_{i,j})_{i,j}$ has no dimension growth). Let $\epsilon>0$ and let $a$ be as in the statement of the lemma. Then there is $j\ge k$ such that  
\begin{align*}
\mathrm{rank}\left(\mathrm{ev}_x(\widetilde \phi_{k,j}(f_{\epsilon}(|a|)))\right)\ge N,
\end{align*}
for all $x\in X_j$. Given $x\in X_j$ let $E_x$ denote the subspace of $\C^{m_j(x)}$ defined by $E_x=\mathrm{ev}_x(\widetilde \phi_{k,j}(f_{\epsilon}(|a|)))\C^{m_j(x)}$. Then by the previous inequality 
$$\dim E_x\ge N>\frac{\dim X_j}{2}.$$

Let $g_\epsilon\colon [0,\infty)\to [0,1]$ be a continuous function such that $g_\epsilon(x)=1$ if $x\in [0,\epsilon]$ and $g_\epsilon(x)=0$ if $x\in [2\epsilon, \infty]$. Then 
\begin{align*}
\||a|g_\epsilon(|a|)\|<2\epsilon,  \quad \mathrm{ev}_x(\widetilde \phi_{k,j})(g_\epsilon(|a|))\xi=\xi
\end{align*}
for all $\xi \in E_x$. It follows that 
\begin{align*}
\|\mathrm{ev}_x(\widetilde \phi_{k,j}(a))\xi\|=\|\mathrm{ev}_x(\widetilde \phi_{k,j}(|a|))\xi\|= \|\mathrm{ev}_x(\widetilde\phi_{k,j}(|a|g_\epsilon(|a|)))\xi\|< 2\epsilon \|\xi\|, 
\end{align*}
for  all $\xi\in E_x\setminus \{0\}$. This shows that $\widetilde\phi_{k,j}(a)\in \widetilde A_j$ satisfies the conditions of Proposition \ref{prop: polar}. Therefore, there is a unitary $v\in \mathrm{U}(\widetilde A_i)$ such that
$$\|\widetilde\phi_{k,j}(a)-v|\widetilde\phi_{k,j}(a)|\|<5\epsilon.$$
Using that *-homomorphisms are contractive we conclude that
$$\|\widetilde \phi_{k,\infty}(a)-u|\widetilde \phi_{k,\infty}(a)|\|<5\epsilon,$$
where $u=\widetilde\phi_{j,\infty}(u)\in \mathrm{U}(\widetilde A)$.
\end{proof}

Let $A$ anf $B$ be C*-algebras a let $\phi\colon A\to B$ be a *-homomorphism. Recall that $\check{\phi}\colon \mathrm{Lat}(A)\to \mathrm{Lat}(B)$ denotes the lattice map defined by $\check\phi(I)=\I(I)$.

\begin{proposition}\label{prop: decomposition}
Let $A=\varinjlim (A_n, \phi_{n,m})$ be an inductive limit of RSH$_0$-algebras. Suppose that $A$ has compact spectrum. Then $A$ can be written as inductive limit $A=\varinjlim (B_n, \psi_{n,m})$, where $B_n$ are RSH$_0$-algebras with compact spectrum and the maps $\psi_{n,m}$ are injective and such that $\check\psi_{n,n+1}(B_n)=B_{n+1}$ (i.e., $\I(\psi_{n,n+1}(B_n))=B_{n+1}$). Moreover, if the system $ (A_n, \phi_{n,m})_{n,m}$ has no dimension growth the new system $ (B_n, \psi_{n,m})_{n,m\in \N}$ can be taken to have no dimension growth.
\end{proposition}
\begin{proof}
Let $\phi_{n,\infty}\colon A_n\to A$ denote the *-homomorphisms given by the inductive limit decomposition of $A$. These maps satisfy that $\phi_{n+1,\infty}\circ \phi_n=\phi_{n,\infty}$ for all $n$. For each $n\ge 1$ let $I_n$ denote the kernel of the map $\phi_n$. Then the kernel of $\phi_n$ is contained in $I_n$ and $\phi_{n}(I_n)\subseteq I_{n+1}$. It follows that the map $\phi_n\colon A_n\to A_{n+1}$ induces an injective map $\overline\phi_n\colon A_n/I_n\to A_{n+1}/I_{n+1}$. Also, it is not difficult to show that
\begin{align*}
\xymatrix{
A_1/I_1 \ar[r]^{\overline\phi_1} & A_2/I_2 \ar[r]^{\overline\phi_2 }& A_3/I_3 \ar[r]^{\overline\phi_3 }&\cdots\ar[r] &A,}
\end{align*}
is an inductive limit decomposition of $A$. By (i) of Proposition \ref{prop: permanence} the C*-algebras $A_n/I_n$, $n=1,2,\cdots$, are RSH$_0$-algebras. Furthermore, the dimension of the total space of $A_n/I_n$ is at most the dimension of the total space of $A_n$. Therefore, we may assume that the maps in the inductive limit decomposition of $A$ are injective.

Now let us use Lemma \ref{idealsIL} to construct a sequence $(k_n)_{n\in \N}$ and C*-algebras $B_n\in A_{k_n}$  such that $B_n=\check\phi_{n,k_n}(A_n)$, $B_n$ has compact spectrum, and $\check\phi_{k_n, k_{n+1}}(B_n)=B_{n+1}$.  Note that if such sequence of C*-algebras exists then
\begin{align*}
\xymatrix{
B_1 \ar[r]^{\phi_{k_1,k_2}} & B_2 \ar[r]^{\phi_{k_2,k_3} }& B_3 \ar[r]^{\phi_{k_3,k_4} }&\cdots\ar[r] &A,}
\end{align*}
is an inductive limit decomposition of $A$ satisfying the conclusions of the proposition. To see this note that the C*-algebras $B_n$ satisfy
\begin{align*}
\overline{\bigcup_{n=1}^\infty \phi_{k_{n},\infty}(B_n)}=\overline{\bigcup_{n=1}^\infty\phi_{k_{n},\infty}(\check\phi_{n,k_n}(A_n))}=\overline{\bigcup_{n=1}^\infty\phi_{n,\infty}(A_n)}=A.
\end{align*}
In addition, the algebras $B_n$ are RSH$_0$-algebras by (iii) of Proposition \ref{prop: permanence} and the dimension of the total space of $B_n$ is at most the dimension of the total space of $A_{k_n}$. Therefore, if the system $(A_n,\phi_n)_{n\in \N}$ has no dimension growth then the system $(B_n, \phi_{k_n, k_{n+1}})_{n\in \N}$ has no dimension growth.  

Let us proceed by induction to construct the sequences $(k_n)_{n\in \N}$ and $(B_n)_{n\in \N}$. By Lemma \ref{idealsIL} applied to the case $n=1$ there is a positive integer $k_1> 2$  such that the ideal $\check\phi_{1,k}(A_1)$ has compact spectrum, and $\check\phi_{2,k}(A_{2})=\check\phi_{1,k}(A_1)$ for all $k\ge k_1$. Set $B_1=\check\phi_{1,k_1}(A_1)$. Suppose that we have constructed the integer $k_n$ and the algebra $B_n$ and let us construct $k_{n+1}$ and $B_{n+1}$. By Lemma \ref{idealsIL} applied to the number $n+1$ there is $k_{n+1}> k_n$ such that the ideal $\check\phi_{n+1,k_{n+1}}(A_{n+1})$ has compact spectrum and $\check\phi_{n+1,k}(A_{n+1})=\check\phi_{n,k}(A_n)$  for all $k\ge k_{n+1}$. Set $B_{n+1}=\check\phi_{n+1,k_{n+1}}(A_{n+1})$. Then,
\begin{align*}
B_{n+1}=\check\phi_{n+1,k_{n+1}}(A_{n+1})=\check\phi_{n, k_{n+1}}(A_n)=\check\phi_{k_n,k_{n+1}}(B_n).
\end{align*}
We have found a C*-algebra $B_{n+1}$ with compact spectrum satisfying  $B_{n+1}=\check\phi_{k_n,k_{n+1}}(B_n)=\I(\phi_{k_n,k_{n+1}}(B_n))$.
This concludes the proof of the proposition. 
\end{proof}

\section{Main results}

\subsection{Reduction of the stable rank}
In this subsection we prove the stable rank results stated in Theorem \ref{thm: mainmain}.

\begin{lemma}\label{lem: invertible}
Let $A$ be a C*-algebra such that every quotient of $A$ is projectionless and every quotient of $\widetilde A$ is finite. Let $a\in \widetilde A$ be such that $a-1\in A$. Let $0<\delta<\min(1,\|a\|)$ and let $f\in \mathrm{C}_0(0,\infty)$ be such that $f(x)>0$ for $x\in (0,\delta)$ and $f(x)=0$ for $x\in [\delta, \infty)$. If $I$ denote the closed two-sided ideal generated by $f(|a|)$, with $|a|=(a^*a)^{\frac 1 2}$, then the image of $a$ in $\widetilde A/I$ is invertible. 
\end{lemma}
\begin{proof}
Since $a-1\in A$ we have $|a|-1\in A$. Therefore, $$f(|a|)\in C^*(|a|-1)\subset A.$$ 
It follows that $I=\I(f(|a|))\subseteq A$. 

Suppose that the image $[a]_I$ of $a$ in the quotient $\widetilde A/I$ is not invertible. Then $[|a|]_I$ is not invertible by the finiteness of $\widetilde A/I$. Since $I=\I(f(|a|))$ we have $f([|a|]_I)=[f(|a|)]_I=0$. Hence, the element $[|a|]_I$ has a gap in its spectrum (here we are using that $\delta<\|a\|$). Now using functional calculus on the element $a$  we can find a nonzero projection $p\in \widetilde A/I$. Note that $p<1_{\widetilde A/I}$ since $0\in \mathrm{sp}([|a|]_I)$. Also,  since $p\in \widetilde A /I$ there are $q\in A/I$ and $\lambda\in \C$ such that $p=\lambda 1_{\widetilde A/I}-q$. This implies using that $p$ is a projection that $\lambda=1$ and $q$ is a projection in $A/I$. But this is impossible since $A/I$ is projectionless and $p<1$.
\end{proof}

\begin{theorem}\label{th: stable rank}
Let $A$ be the inductive limit of a system of RSH$_0$-algebras with no dimension growth. Assume that $A$ has a finite number of ideals. Then the stable rank of $A$ is one if and only if for every ideal $I$ of $A$ the index map $\delta\colon \mathrm{K}_1(A/I)\to \mathrm{K}_0(I)$ is zero.
\end{theorem}
\begin{proof}
First, let us assume that $A$ is an infinite dimensional simple C*-algebra of the form $A=\varinjlim (A_n, \phi_{n,m})$, where $(A_n, \phi_{n,m})_{n,m}$ is a system of RSH$_0$-algebras with no dimension growth.  Since $A$ is simple it has compact spectrum; thus, by Proposition \ref{prop: decomposition} we may assume that the C*-algebras $A_n$ have compact spectrum, that the maps $\phi_{n,m}\colon A_n\to A_m$ are injective, and that $\I(\phi_{n,m}(A_n))=A_m$ for all $m\ge n$.

Suppose that there is a non-zero projection $p\in A$. Since every projection in the inductive limit C*-algebra is the image of a projection from  a finite stage algebra there are $N>0$ and projections $p_n\in A_n$, with $n>N$, such that $\phi_{n,m}(p_n)=p_m$ and $\phi_{n,\infty}(p_n)=p$. The C*-algebras $p_nA_np_n$ are RSH-algebras and $pAp=\varinjlim p_nA_np_n$ (the first statement follows by (iv) of Proposition \ref{prop: permanence}).  Hence, by \cite[Theorem 3.6]{PhillipsIRSH} the stable rank of $pAp$ is one. Since $A$ is simple it is stable isomorphic to $pAp$, thus $\mathrm{tsr}(A)=1$ by \cite[Theorem 3.6]{Rieffel}.

Now suppose that $A$ is projectionless. Then the C*-algebras $A_n$ are projectionless since the maps $\phi_{n,m}$ are injective. Let us consider the unitization $\widetilde A$ of $A$. Then $\widetilde A=\varinjlim (\widetilde A_n,\widetilde \phi_n)$. In order to show that $\mathrm{sr}(A)=1$ it is enough to prove that for every $n\ge 1$ and every element $a\in \widetilde A_n$, with $a-1\in A_n$, the element $\widetilde \phi_{n,\infty}(a)$ can be approximated by invertibles in $\widetilde A$. This is clear since the elements of the form $\lambda\widetilde \phi_{n,\infty}(a)$, with $a$ as above and $0\neq \lambda\in \C$, are dense in $\widetilde A$. 

Let $a\in \widetilde A_n$, for some $n$, be such that $a-1\in A_n$. Set $b=\widetilde \phi_{n,\infty}(a)$. If $b$ is invertible then the claim holds, so let us assume that $b$ is not invertible.  Let $0<\epsilon<\min(1,\|b\|)$ and let $f\in \mathrm{C}_0(0,\infty)$ be such that $f(x)>0$ for $x\in (0,\epsilon)$ and $f(x)=0$ for $x\in (\epsilon, \infty)$. Consider the element $f(|b|)$. By functional calculus $f(|b|)\in C^*(|b|-1)\subset A$. In addition, the element $f(|b|)$ is non-zero (if $f(|b|)=0$ then by Lemma \ref{lem: invertible} applied to $I=\{0\}$ we would get that $b$ is invertible). Therefore, $\I(f(|b|))=A$. 

Let $N\in \N$. Choose $m>n$ such that there are mutually orthogonal nonzero positive elements $a_1, a_2, \cdots, a_N\in \Her(\phi_{n,m}(f(|a|)))$. (If this is not possible, then \cite[Lemma 1.6]{PhillipsIRSH} implies that 
$$\dim(\Her(\phi_{n,m}(f(|a|))))\le (N-1)^2,$$ 
for all $m\ge n$. So, $\dim (\Her(f(|b|))\le (N-1)^2$. This implies that $\Her(f(|b|))$ has a projection, which is not possible since $A$ is projectionless.) Since  $a_1, a_2, \cdots, a_N$ are nonzero and the maps $\phi_{n,m}$ are injective  $\I(\phi_{n,\infty}(a_i))=A$ for all $i$. Hence, by Lemma \ref{lem: idealgen} there is $k\ge m$ such that $\I(\phi_{m, k}(a_i))=A_k$ for all $i$. In particular, this implies that $\phi_{m,k}(a_i)(x)\neq 0$ for every $x\in X_k$ and every $i$, where $X_k$ denotes the total space of $A_k$. Therefore, since the elements  $a_1, a_2, \cdots, a_N$ are mutually orthogonal and $\phi_{m,k}(a_i)\in \Her(\phi_{n,k}(f(|a|)))$ for all $i$ we have
$$\mathrm{rank} (\mathrm{ev}_x(\phi_{n,l}(f(|a|)))\ge N,$$
for all $x\in X_l$ and all $l\ge k$.  Since $N\in \N$ is arbitrary $\phi_{n,\infty}(a)$ can be approximated by invertibles in $\widetilde A$ by Lemma \ref{lem: rank}.

This ends the proof in the  case that $A$ is a simple C*-algebra. The proof of the general case follows by applying inductively \cite[Lemma 3]{Nistor}.
\end{proof}



Let $\mathcal A$ be the class of C*-algebras defined in Definition \ref{def: A}. Then $\mathcal A$ is closed under direct limits by (vi). Also, it is closed under the operation of taking pullbacks of C*-algebras (in fact, by \cite[Proposition 3.4]{Pedersen} this is true for any class of C*-algebras that satisfies (v), (iv), and the second part of (iii)). In particular, this implies that $\mathcal A$ contains the class of RSH$_0$-algebras since $\M_n(\mathrm{C}_0(X))\in \mathcal{A}$ for every $n$ and every locally compact Hausdorff space $X$, by (i) and (ii). The class $\mathcal{A}$ also contains C*-algebras with infinite projections (e.g., the Toeplitz algebra), simple C*-algebras with arbitrary stable rank (e.g., Villadsen's algebras), and even simple C*-algebras that can not be classified using the Elliott invariant (e.g., Tom's algebras). 

 
\begin{proposition}\label{thm: type I}
Every separable type I C*-algebra belongs to $\mathcal{A}$. 
\end{proposition} 
\begin{proof} 
By the previous remark $\mathcal{A}$ contains the class of RSH$_0$-algebras. Hence, by Theorem \ref{thm: ASH} it also contains the class of separable subhomogeneous C*-algebras. Now, let us show that $\mathcal{A}$ contains the class of separable continuous trace C*-algebras. Let $A$ be a separable continuous trace C*-algebra then the set of elements of $a\in A$ with uniformly bounded trace (i.e., there exists $n\in \N$ such that $\mathrm{rank}(\pi(a))\le n$ for all irreducible representation $\pi$) form a dense self-adjoined ideal of $A$ (not necessarily closed). Note that for such elements the hereditary algebra $\overline{a^*Aa}$ is a subhomogeneous algebra. Also, since this ideal is dense in $A$ and $A$ is separable, $A$ can be written as a direct limit of such hereditary algebras. Therefore, $A\in \mathcal{A}$.

Suppose now that $A$ is an arbitrary separable type I C*-algebra. By \cite[Theorem 6.2.11]{PedersenBook} and by the separability of $A$, $A$ has a composition series $(I_n)_{n=0}^\infty$ such that the quotients $I_{n+1}/I_n$ are continuous trace C*-algebras. Hence, by (v) of the above definition $I_n\in \mathcal{A}$ for all $n$. Since $A=\overline{\bigcup_n I_n}$ we deduce that $A\in \mathcal{A}$ since $\mathcal{A}$ is closed under direct limits. 
\end{proof}

\begin{lemma}\label{lem: path}
For every $\epsilon>0$ there is $\delta>0$ such that: if $A$ is a C*-algebra satisfying $\mathrm{U}(B)=\mathrm{U}_0(B)$ for every hereditary sub-C*-algebra $B$ of $A$, and if $a\in  \widetilde A$ is a positive element with $a-1\in A$ and $\|a\|\le M$,  and $u\in \widetilde A$ is a unitary with $u-1\in A$ satisfying
$$\|ua-a\|<\delta,$$
then there exists a path of unitaries $(u_t)_{t\in [0,1]}$ in $\widetilde A$ with $u_t-1\in A$, $u_0=u$, and $u_1=1$, such that
$$\|u_ta-a\|<\epsilon,$$
for all $t\in [0,1]$.
\end{lemma}
\begin{proof}
The universal C*-algebra generated by elements $a$ and $u$ satisfying the relations:
\begin{align}\label{eq: relations}
u^*u=uu^*=1,\quad 0\le a\le M\cdot 1,\quad ua=a,
\end{align}
is isomorphic to the algebra of complex-value continuous functions over the graph consisting of a circle with an interval attached to it at a point. Since this C*-algebra is weakly semiprojective the relations mention above are weakly stable. 

Let $0<\epsilon<1$ and let $\epsilon'=\min(\frac{\epsilon}{2},\frac{\epsilon}{3M+1})$. Then by the weak stability of the relations \eqref{eq: relations} there exists $0<\delta<\epsilon'$ such that: if $a\in \widetilde A$ is a positive element and $u\in \mathrm{U}_0(A)$ satisfy $\|ua-a\|<\delta$ then there are $a'\in \widetilde A$, with $0\le a'\le M\cdot 1$, and $u'\in \mathrm{U}_0(\widetilde A)$ such that 
\begin{align*}
u'a'=a', \quad \|u-u'\|<\epsilon',\quad \|a-a'\|<\epsilon'.
\end{align*}
Let $\pi\colon \widetilde A\to \C$ be the quotient map. Then we have
$$\pi(u')\pi(a')-\pi(a')=0,\quad |\pi(a')-1|<\epsilon'<\frac 1 2,\quad |\pi(u')-1|<\epsilon'<\frac 1 2.$$
It follows that $\pi(u')=1$. In other words $u'-1\in A$ and 
$$\left\|a-\frac{a'}{\pi(a')}\right\|\le \|a-a'\|+\left\|\frac{a'}{\pi(a')}(\pi(a')-1)\right\|<\epsilon'+2M\epsilon'=(2M+1)\epsilon'.$$
If we replace $a'$ by $\frac{a'}{\pi(a')}$ then we get a positive element $a'\in \widetilde A$ and a unitary $u'\in \widetilde A$ in the connected component of the identity such that
\begin{align}\label{eq: ua}
\begin{aligned}
& a'-1, u'-1\in A, \quad u'a'=a', \\
& \|u-u'\|<\epsilon',\quad \|a-a'\|<(2M+1)\epsilon'.
\end{aligned}
\end{align}
Write $u'=x+1$, with $x\in A$. Then $xa'=0$. Also, since $x$ is a normal element $x\in \overline{x^*Ax}$. In particular, $u'$ is unitary of the sub-C*-algebra generated by $\overline{x^*Ax}$ and the unit of $A$. Note that this C*-algebra can be identified with $(\overline{x^*Ax})\widetilde{\,\,}$. Hence, by hypothesis there is a path of unitaries $(v_t)_{t\in [0,1]}$ in $(\overline{x^*Ax})\widetilde{\,\,}$ such that $v_0=1$ and $v_1=u'$. Moreover, this path of unitaries may be taken such that $v_t-1\in \overline{x^*Ax}$ for all $t$. Note that with this choice $(v_t)_{t\in [0,1]}$ satisfies $v_ta'-a'=0$. Now it follows that
\begin{align}\label{eq: v}
\|v_ta-a\|\le \|v_t(a-a')\|+\|a'-a\|<2\epsilon'\le \epsilon.
\end{align}

Let $z_t=tu+(1-t)u'$, with $t\in [0,1]$. Then by \eqref{eq: ua}
$$\|u-z_t\|=\|(1-t)(u-u')\|<\epsilon'<\frac 1 2.$$
Hence $z_t$ is invertible for all $t$ and by \cite[Lemma 1]{Ciuperca-Elliott-Santiago} we have
$$\|1-|z_t|\|\le \|1-z_t^*z_t\|^{\frac 1 2}\le (\|(u-z_t)^*u\|+\|z_t^*(u-z_t)\|)^{\frac 1 2}<(\epsilon'+\epsilon')^{\frac 1 2}<2\epsilon'.$$
Since $z_t$ is invertible it has a polar decomposition $z_t=w_t|z_t|$ with $w_t$ a unitary in $\widetilde A$. Note that the path of unitaries $(w_t)_{t\in [0,1]}$ satisfies $w_0=u'$, $w_1=u$, and $w_t-1\in A$.  Also, by the previous inequalities we have
$$ \|u-w_t\|\le \|u-w_t|z_t|\|+\|w_t(|z_t|-1)\|<\epsilon'+2\epsilon'=3\epsilon'.$$
This implies that
\begin{align}\label{eq: w}
\|w_ta-a\|=\|(v_t-u)a\|+\|ua-a\|<3M\epsilon'+\delta\le (3M+1)\epsilon'\le \epsilon.
\end{align}
The proof of the lemma now follows by taking
\begin{align*}
u_t=
\begin{cases}
v_{2t},\quad \text{if}\quad t\in [0,\frac 1 2],\\
w_{2t-1},\quad \text{if}\quad t\in [\frac 1 2, 1].
\end{cases}
\end{align*}
Indeed, $u_0=1$, $u_1=u$, $u_t-1\in A$, and by \eqref{eq: v} and \eqref{eq: w} $\|u_ta-a\|<\epsilon$ for all $t\in [0,1]$.
\end{proof}

Let $X$ be a compact Hausdorff space an let $V$ be an open subset of $X$. We denote by $\partial V$, $\overline{V}$, and $X\setminus V$ the boundary, the closure, and the complement in $X$ of the set $V$. Let $B$ be a C*-algebra. In the following we use the natural identifications:
\begin{align*}
&\mathrm{C}(X, B)\cong \mathrm{C}(X)\otimes B,\\
&\mathrm{C}(X,B)/\mathrm{C}_0(V,B)\cong \mathrm{C}(X\!\setminus\! V, B),\\
&\mathrm{C}(X,B)\widetilde{\,\,}/\mathrm{C}_0(V,B)\cong \mathrm{C}(X\!\setminus\! V, B)\widetilde{\,\,}.
\end{align*}
Also, if $B$ is non-unital we identify $\mathrm C(X, B)\widetilde{\,\,}$ with the subalgebra of $\mathrm C(X, \widetilde B)$ generated by $\mathrm C(X, B)$ and the unit of $\mathrm C(X, \widetilde B)$. In this way, the image of an element $a\in \mathrm C(X, B)\widetilde{\,\,}$ in the quotient by the ideal $\mathrm C(V, B)$ may be identified with the restriction $a_{X\setminus V}$ of $a$ to the closed set $X\!\setminus\! V$.

\begin{lemma}\label{lem: appinvertibles}
Let $B$ be a C*-algebra and let $X=[0,1]^k$, for some positive integer $k$. Let $a$ be an element of $\mathrm{C}(X, B)\widetilde{\,\,}$ such that $a-1\in \mathrm{C}(X, B)$. Suppose that for every $\epsilon>0$ there are open subsets $U, V, W\subseteq X$, with
$\overline U\subset V\subset \overline V\subset W$,
such that the stable rank of $\mathrm{C}(\partial V, B)$ is one, every ideal of $\mathrm{C}(\partial V, B)$ has trivial $\mathrm{K}_1$-group, $a_{X\setminus U}$ is in the closure of the invertible elements of  $\mathrm{C}(X\setminus U, B)\widetilde{\,\,}$, and there is a unitary $u\in \mathrm{C}(\overline{W}, B)\widetilde{\,\,}$ such that
$$\|a_{\overline W}-u|a_{\overline{W}}|\|<\epsilon.$$
Then $a$ is in the closure of the invertible elements of $\mathrm{C}(X, B)\widetilde{\,\,}$.
\end{lemma}
\begin{proof}
Let $a\in \mathrm{C}(X, B)\widetilde{\,\,}$ with $a-1\in \mathrm{C}(X, B)$ and let $\epsilon>0$. Let us show that there is a unitary $w\in \mathrm{C}(X, B)\widetilde{\,\,}$ such that $\|a-w|a|\|<\epsilon$. (This clearly implies that $a$ is the closure of the invertible elements of $\mathrm{C}(X, B)\widetilde{\,\,}$  since $\epsilon$ is arbitrary and $|a|$ can be approximated by invertible elements.)

Choose $0<\delta<\epsilon$ satisfying the statement of Lemma \ref{lem: path} for $\frac \epsilon 2$. Choose open subsets $U, V, W\subseteq X$ satisfying the assumptions of the lemma for the number $\delta/2$. Then there are unitaries $u_0\in \mathrm{C}(\overline W, B)\widetilde{\,\,}$ and $u_1\in\mathrm{C}(X\setminus U, B)\widetilde{\,\,}$ such that
\begin{align}\label{eq: u0u1}
\|a_{\overline W}-u_0|a_{\overline W}|\|<\frac \delta 2,\quad \|a_{X\setminus U}-u_1|a_{X\setminus U}|\|<\frac \delta 2.
\end{align}
Moreover, by  a simple perturbation argument we may assume that $u_0-1\in \mathrm{C}(\overline W, B)$ and $u_1-1\in \mathrm{C}(X\setminus U, B)$. 

Now let us restrict the inequalities above to the boundary $V$. Then, it follows that
$$\|(u_0)_{\partial V}|a_{\partial V}|-(u_1)_{\partial V}|a_{\partial V}|\|<\delta.$$ 
By choice of $V$, the stable rank of $\mathrm{C}(\partial V, B)$ is one, and the $\mathrm{K}_1$-group of every ideal of $\mathrm{C}(\partial V, B)$ is trivial. Hence, every hereditary sub-C*-algebra of $\mathrm{C}(\partial V, B)$ has stable rank one and trivial $\mathrm{K}_1$-group. This implies, by \cite[Theorem 2.10]{Rieffel}, the the unitary group of  every hereditary sub-C*-algebra of $\mathrm{C}(\partial V, B)$ is connected. It follows now by the preceding inequality and Lemma \ref{lem: path}, applied to the element $|a_{\partial V}|$ and to the unitary $(u_0)_{\partial V}^*(u_1)_{\partial V}$, that there is a path of unitaries $u(t)\colon [0,1]\to \mathrm{C}(\partial V, B)\widetilde{\,\,}$, with $u(t)-1\in \mathrm{C}(\partial V, B)$, such that
$u(0)=(u_0)_{\partial V}$, $u(1)=(u_1)_{\partial V}$, and
$$\|u(t)|a_{\partial V}|-u_1|a_{\partial V}|\|<\frac \epsilon 2,$$
for all $t\in [0,1]$.
Using this inequality and the second inequality of \eqref{eq: u0u1} restricted to the boundary of $V$ we get
$$\|a_{\partial V}-u(t)|a_{\partial V}|\|<\epsilon,$$
for all $t\in [0,1]$. Let $d\colon X\to [0,\infty)$ be a metric in $X$ and let 
$$F_n=\{x\in X \mid d(x, \partial V)\le 1/n\}.$$
Then $\bigcup_nF_n=\partial V$ and $\varinjlim \mathrm{C}(F_n, B)\widetilde{\,\,}=\mathrm{C}(\partial V, B)\widetilde{\,\,}$, where the connecting *-homomorphisms are the obvious quotient maps. Since any path of unitaries in an inductive limit C*-algebra can be lifted approximately to finite stage C*-algebra there are $F_n\subseteq \overline V\cap (X\setminus U)$ and a path of unitaries $v(t)\colon [0,1]\to \mathrm{C}(F_n, B)\widetilde{\,\,}$ such that $v(0)=(u_0)_{F_n}$, $v(1)=(u_1)_{F_n}$, and 
\begin{align}\label{eq: Fn}
\|a_{F_n}-v(t)|a_{F_n}|\|<\epsilon,
\end{align}
for all $t\in [0,1]$. Moreover, since $u(t)-1\in \mathrm{C}(\partial V, B)\widetilde{\,\,}$ we may choose $v(t)$ such that $v(t)-1\in \mathrm{C}(F_n, B)$. Let $w$ be the element defined by
\begin{align*}
w(x)=
\begin{cases}
 u_0(x),\quad &\text{if }\,\,x\in V\setminus F_n,\\
u_1(x) , \quad &\text{if }\,\, x\in X\setminus (V\cup F_n),\\
 v(g(x))(x), \quad & \text{if }\,\, x\in F_n,
\end{cases}
\end{align*}
where $g\colon F_n\to [0,1]$ is given by
$$g(x)=\frac{d\left(x, \overline{V\setminus F_n}\right)}{\max\left\{d\left(x,\overline{X\setminus (V\cup F_n)}\right), d\left(x,\overline{V\setminus F_n}\right)\right\}}.$$
Note that $g(x)$ is continuous and $g(x)=0$ if $x\in F_n\cap \overline{V\setminus F_n}$, and $g(x)=1$ if $x\in F_n\cap \overline{X\setminus (V\cup F_n)}$. Hence, $w$ is a well defined and it is a unitary of $\mathrm{C}(X, B)\widetilde{\,\,}$ since $w(x)-1\in B$ for all $x\in X$. It follows now by equations \eqref{eq: u0u1} and \eqref{eq: Fn} that
$$\|a-w|a|\|<\epsilon.$$
\end{proof}

Recall that a C*-algebra $B$ is locally contained in a class of C*-algebras $\mathcal A$ if for every $\epsilon>0$ and every finite subset $F$ of $B$ there exists a C*-algebra $A\in \mathcal A$ and a *-homomorphism $\phi\colon A\to B$ such that the distance from $x$ to $\phi(A)$ is less than $\epsilon$ for every $x\in F$.
 
\begin{theorem}\label{thm: main1}
Let $A$ be the class of C*-algebras defined in Definition \ref{def: A}. Let $B$ be a simple inductive limit of a system of RSH$_0$-algebras with no dimension growth. The following statements hold:

\begin{enumerate}
\item If $B$ is projectionless and $\mathrm{K}_0(B)=\mathrm{K}_1(B)=0$ then $\mathrm{sr}(A\otimes B)=1$ for every C*-algebra $A\in \mathcal{A}$.
\item If $B$ is either projectionless or it is unital with no nonzero projection but its unit and $\mathrm{K}_1(B)=0$ then $\mathrm{sr}(A\otimes B)=1$ for every C*-algebra $A$ that is approximately contained in the class of RSH$_0$-algebras with  1-dimensional spectrum.
\end{enumerate}
\end{theorem}
\begin{proof}
(i) Let $\mathcal{D}$ be the subclass of $\mathcal{A}$ consisting of the C*-algebras $A\in \mathcal{A}$ such that $\mathrm{sr}(A\otimes B)=1$. Then by the results of \cite{Rieffel} (i.e., Theorems 6.1, 4.4, 6.4, 4.3, and the proof of Theorem 5.1) and by \cite[Lemma 3]{Nistor} it follows that (ii), (iii), (iv), (v), and (vi) of Definition \ref{def: A} holds if we replace $\mathcal{A}$ by $\mathcal{D}$. Hence, $\mathcal{D}$ agrees with $\mathcal{A}$ if and only if  $\mathrm{C}_0(X)\in \mathcal{D}$, for every locally compact Hausdorff space $X$.


Let $X$ be a locally compact Hausdorff space. Let us show that the stable rank of $\mathrm{C}_0(X, B)$ is one. First we consider the case of compact metric spaces $X$ of finite covering dimension.  The proof in this case will proceed by induction on the covering dimension of $X$.
If the covering dimension of $X$ is zero then $\mathrm{C}(X)$ is an AF-algebra. Hence, $\mathrm{sr}(\mathrm{C}(X, B))=1$ since by Theorem \ref{th: stable rank} the stable rank of $B$ is one. Now let us assume that $\mathrm{sr}(\mathrm{C}(X, B))=1$ for every compact metric space $X$ of covering dimension at most $n-1$.

Let $B=\varinjlim (B_i, \phi_{i,j})$ be as in the statement of the theorem. By Proposition \ref{prop: decomposition} we may assume that the C*-algebras $B_i$ have compact spectrum, that the *-homomorphisms $\phi_{i,j}$ are injective, and that $\I(\phi_{i,i+1}(B_i))=B_{i+1}$ for all $i$. Also, we have
$$\mathrm{C}(X, B)\widetilde{\,\,}=\varinjlim\left (\mathrm{C}(X, B_i)\widetilde{\,\,}, (\id_X\otimes \phi_{i,j})\widetilde{\,\,}\right),$$
where $\id_X:\mathrm{C}(X)\to \mathrm{C}(X)$ denotes the identity map. 
In particular, this implies that any element of $\mathrm{C}(X, B)\widetilde{\,\,}$ can be approximated by the images in $\mathrm{C}(X, B)\widetilde{\,\,}$ of elements of the C*-algebras  $\mathrm{C}(X, B_i)\widetilde{\,\,}$. In addition, any element of $\mathrm{C}(X, B_i)\widetilde{\,\,}$ can be approximated by elements of the form $\lambda (b+1)$, with $\lambda\in \C\setminus\{0\}$ and $b\in \mathrm{C}(X, B_i)$. Therefore, $\mathrm{sr}(\mathrm{C}(X,B))=1$ if for every $i\ge 1$ and for every element $a\in \mathrm{C}(X,A_i)\widetilde{\,\,}$, with $a-1\in \mathrm{C}(X, A_i)$, the element $(\id_X\otimes\phi_{n,\infty})\widetilde{\,\,} (a)$ can be approximated by invertible elements of $\mathrm{C}(X, B)\widetilde{\,\,}$.

Let $X=[0,1]^n$. Let $a\in \mathrm{C}(X, B_i)\widetilde{\,\,}$, for some $i$, be such that $a-1\in \mathrm{C}(X, B)$ and such that $(\id_X\otimes \phi_{i,\infty})\widetilde{\,\,}(a)$ is not invertible. Let $\epsilon>0$. Consider the elements
\begin{align}\label{eq: a'bb'}
\begin{aligned}
a'=(\id_X\otimes\phi_{i,\infty})\widetilde{\,\,}(a),\quad b=f(|a|),\quad b'=(\id_X\otimes\phi_{i,\infty})\widetilde{\,\,}(b),
\end{aligned}
\end{align}
where $|a|$ denotes the element $(a^*a)^{\frac 1 2}$, $0< \epsilon<\min(1,\|a\|)$, and $f\in \mathrm C_0(0,\infty)$ is such that $f(x)>0$ for $x\in (0,\epsilon)$ and $f(x)=0$ for $x\in (\epsilon, \infty)$. By functional calculus $b\in \mathrm{C}(X,B_i)$ and as a result $b'\in \mathrm{C}(X,B)$. As a consequence, the ideal generated by $b'$ has the form $\mathrm{C}_0(Y,B)$ for some open subset $Y$ of $X$ (to deduce this we are also using the simplicity of $B$). This implies by Lemma \ref{lem: invertible}, applied to the element $a'$ and to the ideal $\mathrm{C}_0(Y,B)$, that $a'_{X\setminus Y}\in \mathrm{C}(X\setminus Y)$ is invertible.  Choose an increasing sequence of open subsets $(Y_j)_{j=1}^\infty$ such that $\overline{Y_j}\subset Y_{j+1}$, $Y=\bigcup_j Y_j$, and the covering dimension of the boundary $\partial Y_j$ of $Y_j$ is at most $n-1$ for all $j$. It follows that, 
\begin{align*}
\mathrm{C}(X\setminus Y, B)\widetilde{\,\,}=\varinjlim \mathrm{C}(X\setminus Y_j, B)\widetilde{\,\,}.
\end{align*}
Hence, since $a'_{X\setminus Y}$ is invertible in $\mathrm{C}(X\setminus Y, B)\widetilde{\,\,}$ there is $j$ such that $a'_{X\setminus Y_k}$ is invertible in $\mathrm{C}(X\setminus Y_k, B)\widetilde{\,\,}$ for all $k\ge j$. Set 
\begin{align}\label{eq: UVW}
U=Y_j, \quad V=Y_{j+1},\quad W=Y_{j+2}.
\end{align} 
Then $a'_{X\setminus U}$ is invertible and $\overline U\subset V\subset \overline V\subset W\subset \overline W\subset U$.

The element $b'$ generates the ideal $\mathrm{C}_0(U, B)$. Therefore, $b'_{\overline{W}}$ generates $\mathrm{C}(\overline{W}, B)$ as an ideal.
This implies using that $\overline W$ is compact that 
$$K=\min_{x\in \overline W}\|b'_{\overline W}(x)\|\neq 0.$$
For each $1\le j\le N$ let $g_j$ be positive continuous function with support on the interval $\left(\frac{(j-1)K}{N}, \frac{jK}{N}\right)$. Set $b'_j=g_j(b'_{\overline W})$. Then $b'_j(x)\neq 0$ for every $x\in \overline W$ (Note that $b_j'(x)=0$ for some $x\in \overline W$ implies that $b_j(x)\in B$ has a gap in its spectrum. In turn, this implies that $B$ contains a nonzero projection which contradics the assumption on $B$.)  Also, the elements $(b'_i)_{i=1}^N$ are mutually orthogonal.

The closed subset $\overline W\subset [0,1]^{n}$ has covering dimension at most $n$. Hence, the C*-algebra $\mathrm{C}(\overline{W},B)$ can be written as and the inductive limit of a system of RSH$_0$-algebras with no dimension growth. Specifically,
\begin{align}\label{eq: limiteind}
\mathrm{C}(\overline{W},B)=\varinjlim (\mathrm{C}(\overline{W},B_j),\id_{\overline{W}}\otimes\phi_{j,k}).
\end{align}
Note that the C*-algebras $\mathrm{C}(\overline{W},B_j)$ have compact spectrum, the homomorphisms $\id_{\overline{W}}\otimes\phi_{j,k}$ are injective, and the image of $\id_{\overline{W}}\otimes\phi_{j,j+1}$ generates $\mathrm{C}(\overline{W},B_{j+1})$ as an ideal for all $j$.

Let $b_1,b_2,\cdots, b_N\in \mathrm{C}(\overline W, B_i)$ be the elements defined by $b_j=g_j(b_{\overline W})$. Then by the definition of $b'$ (see \eqref{eq: a'bb'}) we have 
\begin{align*}
b_j'=(\id_{\overline{W}}\otimes\phi_{i,\infty})(b_j).
\end{align*}
By Lemma \ref{lem: idealgen}, applied to the inductive limit \eqref{eq: limiteind} and to the elements $(b_j)_{j=1}^N$, there exists $k>i$ such that 
$$\I((\id_{\overline{W}}\otimes\phi_{i,l})(b_j))=\mathrm{C}(\overline W, B_l),$$  
for all $j$ and all $l\ge k$. This implies that 
$$\mathrm{ev}_{x,y}((\id_{\overline{W}}\otimes\phi_{i,l})(b_j))\neq 0,$$
for all $1\le j\le N$, $l\ge k$, $x\in \overline W$, and $y$ in the total space of $B_l$. As a consequence, since $(b_i)_{i=1}^N$ are pairwise orthogonal elements of the C*-algebra generated by $b_{\overline W}$ we get
$$\mathrm{rank}\left(\mathrm{ev}_{x,y}\left((\id_{\overline{W}}\otimes\phi_{i,l})(b_{\overline{W}})\right)\right)\ge N,$$
for all $l\ge k$, $x\in \overline{W}$, and $y$ in the total space of $B_l$.
This implies by Lemma \ref{lem: rank} that there exists a unitary
$u\in \mathrm{C}(\overline W, B)\widetilde{\,\,}$ such that
$$\left\|a'_{\overline W}-u|a'_{\overline W}|\right\|<5\epsilon.$$
We have shown that given $\epsilon>0$ there are open subsets $U, V, W\subset X$ such that
$\overline{U}\subset V\subset \overline V\subset W$, the covering dimension of $\partial V$ is at most $n-1$, $a'_{X\setminus U}$ is invertible, and $a'_{\overline W}$ satisfies the inequality above. Also, note that every ideal of $\mathrm{C}(\partial V, B)$ has trivial $\mathrm{K}_1$-group and by the induction hypothesis $\mathrm{sr}(\mathrm{C}(\partial V, B))=1$. Therefore, by Lemma \ref{lem: appinvertibles} the element $a'$ is invertible. This shows that $\mathrm{sr}(\mathrm{C}([0,1]^n, B)=1$.

Let $X$ be a CW-complex of dimension at most $n$. Then $\mathrm{C}(X)$ admits a recursive subhomogeneous decomposition with base spaces homeomorphic to cubes of dimension at most $n$. This implies by \cite[Theorem 3.9]{Pedersen} and \cite[Proposition 4.1 (i)]{Brown-Pedersen} that $\mathrm{sr}(\mathrm{C}(X,B))=1$. Now, suppose that $X$ is an arbitrary compact metric space of covering dimension at most $n$. Then, by \cite[Theorem 1.13.2]{Engelking}, $X$ is homeomorphic to an inverse limit of CW-complexes of dimension at most $n$. This implies that $\mathrm{C}(X,B)$ can be written as a direct limit of a sequence of C*-algebras of the form $\mathrm{C}(Y_i, B)$, where each space $Y_i$ is a CW-complex of dimension at most $n$. It follows now by \cite[Theorem 5.1]{Rieffel} that the stable rank of $\mathrm{C}(X,B)$ is one. This concludes the proof by induction.

By \cite[Theorem 3.3.5]{Engelking} every compact Hausdorff space $X$ is homeomorphic to an inverse limit of compact metric spaces of finite covering dimension. Hence, as above, we conclude that $\mathrm{sr}(\mathrm{C}(X,B))=1$ by \cite[Theorem 5.1]{Rieffel} (this theorem is stated for direct limits of sequences of C*-algebras but it is easy to see that it holds for arbitrary direct limits). Now, let $X$ be an arbitrary locally compact Hausdorff space. Then $\mathrm{C}_0(X,B)$ is an ideal of $\mathrm{C}(\widetilde X,B)$, where $\widetilde X$ denotes the one-point compactification of $X$. Since $\widetilde X$ is compact and Hausdorff $\mathrm{sr}(\mathrm{C}(\widetilde{X}, B))=1$. Therefore, by \cite[Theorem 4.4]{Rieffel}, $\mathrm{sr}(\mathrm{C}_0(X,B))=1$.

(ii) Let $B$ be a unital C*-algebra containing no nontrivial projection but its unit. Let $a\in B$ be a non-invertible positive element of $B$. Then the hereditary sub-C*-algebra $\overline{aBa}$ is projectionless. In addition, $B$ and $\overline{aBa}$ are stable isomorphic since $B$ is simple. Therefore, $\mathrm{K}_1(\overline{aBa})=0$ and by \cite[Theorem 6.4]{Rieffel} the stable rank of the tensor product of $B$ with a C*-algebra $A$ is one if and only if the stable rank of the tensor product of $\overline{aBa}$ with $A$ is one. This shows that it is sufficient to prove (ii) for projectionless C*-algebras.

Let $B$ be a projectionless C*-algebra as in (ii). Then by repeating the first step of the induction in the proof of (i) we conclude that the stable rank of $\mathrm{C}([0,1], B)$ is one. (Note that in the proof of (i) the assumption on the K-groups of $B$ was only used to assure that every ideal of $\mathrm{C}(\partial V, B)$ has trivial $\mathrm{K}_1$-group. This clearly holds if the $\mathrm{K}_1$-group of $B$ is trivial and $V$ is an open subset of $[0,1]$.) Also, as in the proof of (i) this implies that the stable rank of $\mathrm{C}_0(X)\otimes B$ is one for every locally compact 1-dimensional space $X$. Now, using \cite[Theorem 6.1]{Rieffel}, we get that the stable rank of $\mathrm{C}([0,1], \M_n)\otimes B$ is one for every $n$. Therefore, by \cite[Proposition 4.1 (i)]{Brown-Pedersen} the stable rank of $A\otimes B$ is one for any RSH$_0$-algebra with 1-dimensional spectrum. The case that $A$ is locally contained in the class of RSH$_0$-algebra with 1-dimensional spectrum follows using the ideas of the proof of \cite[Theorem 5.1]{Rieffel} and that the C*-algebra $A\otimes B$ is approximately contained in the class of C*-algebras that are tensor products of $B$ with RSH$_0$-algebras with one dimensional spectrum.
\end{proof}

\begin{corollary}
If $A$ is the inductive limit of separable type I C*-algebras and $B$ is as in Theorem \ref{thm: main1} then $\mathrm{sr}(A\otimes B)=1$.
\end{corollary}

An important example of a simple stably projectionless C*-algebra with trivial K-groups is the C*-algebra $\mathcal W$ constructed in \cite{Jacelon}. This C*-algebra is an inductive limit of RSH$_0$-algebra; thus, the first part of Theorem \ref{thm: main1} applies to $\mathcal W$. In other words, the following statement hold:    
\begin{corollary}
Let $A$ be a C*-algebra in $\mathcal{A}$. Then $\mathrm{sr}(A\otimes \mathcal W)=1$.
\end{corollary}

Another important example of a simple C*-algebra for which Theorem \ref{thm: main1} applies is the Jiang-Su algebra $\mathcal{Z}$. This C*-algebra is constructed as an inductive limit of RSH-algebras and is such that $\mathrm{K}_1(\mathcal{Z})=0$ and $\mathrm{K}_0(\mathcal{Z})=\Z$. For this C*-algebra we can proof the following:
 
\begin{proposition}\label{prop: Jiang-Su}
The following statements hold:
\begin{enumerate}
\item If $A$ is a C*-algebra that is locally contained in the class of RSH$_0$-algebras with 1-dimensional spectrum then $\mathrm{sr}(A\otimes \mathcal Z)=1$.
\item If $0\to A\to B\to C\to 0$ is an exact sequence of C*-algebras with nonzero index map $\delta\colon \mathrm{K}_1(A/I)\to \mathrm{K}_0(I)$ then $\mathrm{sr}(B\otimes \mathcal{Z})\neq 1$.
\item If $X$ is CW-complex then 
\begin{align*}
\mathrm{sr}(\mathrm{C}(X)\otimes \mathcal{Z})=
\begin{cases}
1 \quad\text{if}\quad\mathrm{dim}(X)\le 1,\\
2 \quad\text{if}\quad\mathrm{dim}(X)>1.
\end{cases}
\end{align*}

\end{enumerate}
\end{proposition}
\begin{proof}
The first part of proposition is a corollary of the second part of Theorem \ref{thm: main1}. Let us now prove (ii) and (iii).

(ii) Let us proceed by contradiction. Suppose that $\mathrm{sr}(B\otimes \mathcal{Z})=1$. Then $\mathrm{sr}(A\otimes \mathcal{Z})=\mathrm{sr}(C\otimes \mathcal{Z})=1$. Using the Kunneth formula and that $\mathrm{K}_0(\mathcal{Z})=\Z$ and $\mathrm{K}_1(\mathcal{Z})=0$ we get the following commutative diagram involving the respective index maps:
\begin{equation*}
\xymatrix{\mathrm{K}_1(C\otimes \mathcal{Z}) \ar[d]^\cong \ar[r]^{\delta'} & \mathrm{K}_0(A\otimes \mathcal{Z})\ar[d]^\cong \\ 
\mathrm{K}_1(C)\ar[r]^{\delta} & \mathrm{K}_0(A)}
\end{equation*}
By assumption $\delta$ is nonzero thus $\delta'$ is nonzero. But this contradicts the fact that $\mathrm{sr}(B\otimes Z)=1$ by \cite[Lemma 3]{Nistor}.

(iii) By \cite[Proposition 1.5]{Sudo} the stable rank of $\mathrm{C}(X)\otimes \mathcal{Z}$ is at most two for any CW-complex $X$. If  $\mathrm{dim} (X)\le 1$ then $\mathrm{sr}(\mathrm{C}(X)\otimes \mathcal Z)=1$ by (i). Now let us consider the case $\mathrm{dim}(X)>1$. We will proceed by contradiction. Suppose that $\mathrm{sr}(\mathrm{C}(X)\otimes \mathcal{Z})=1$. Since the dimension of $X$ is greater than one $X$ contains a closed subspace homeomorphic to the unit disc $\mathbb{D}$. It follows that both $\mathrm{C}(\overline{\mathbb{D}})\otimes \mathcal Z$ and $\mathrm{C}(\mathbb{T})\otimes \mathcal{Z}$ are quotients of $\mathrm{C}(X)\otimes \mathcal{Z}$ and that $\mathrm{C}(\mathbb{D})\otimes \mathcal{Z}$ is an ideal of $\mathrm{C}(X)\otimes\mathcal{Z}$.  Hence,
$$\mathrm{sr}(\mathrm{C}(\overline{\mathbb{D}})\otimes \mathcal{Z})=\mathrm{sr}(\mathrm{C}_0(\mathbb{D})\otimes \mathcal{Z})=\mathrm{sr}(\mathrm{C}(\mathbb{T})\otimes \mathcal{Z})=1.$$ 
Consider the three term exact sequence:
$$0\to \mathrm{C}_0(\mathbb{D})\to\mathrm{C}(\overline{\mathbb{D}})\to \mathrm{C}(\mathbb{T})\to 0.$$ 
Then the index map $\delta\colon \mathrm{K}_1(\mathrm{C}(\mathbb{T}))\to \mathrm{K}_0(\mathrm{C}(\mathbb{D}))$ is non-zero. This implies that the respective index map corresponding to the exact sequence
 $$0\to \mathrm{C}_0(\mathbb{D})\otimes \mathcal{Z}\to\mathrm{C}(\overline{\mathbb{D}})\otimes \mathcal{Z}\to \mathrm{C}(\mathbb{T})\otimes \mathcal{Z}\to 0,$$
is nonzero. But this contradicts the fact that $\mathrm{sr}(\mathrm{C}(\overline{\mathbb{D}})\otimes \mathcal{Z})=1$ by \cite[Lemma 3]{Nistor}. 
\end{proof}

\begin{proposition}
Let $A$ and $B$ be C*-algebras with $B$ as in Theorem \ref{thm: main1}. Suppose that a quotient of $A\otimes \mathcal{K}$ contains a properly infinite projection. Then $\mathrm{sr}(A\otimes B)\neq 1$.
\end{proposition}
\begin{proof}
Suppose that $I$ is an ideal of $A$ such that the quotient C*-algebra $(A/I)\otimes \K$ contains a properly infinite projection. In other words, there is an embedding of $\mathcal{O}_\infty$ into $(A/I)\otimes \K$.  By the classification theorem of Kirchberg and Phillips $\mathcal{O}_\infty\otimes B\cong \mathcal{O}_2\otimes \K$; thus,  $\mathcal{O}_2$ embeds into $(A/I)\otimes B\otimes \K$. It follows that, $\mathrm{sr}((A/I)\otimes B\otimes\K)\neq 1$. Now, using that $\mathrm{sr}((A/I)\otimes B\otimes \K)\le \mathrm{sr}(A\otimes B\otimes \K)$ and that $\mathrm{sr}(A\otimes B\otimes \K)=1$ if and only if $\mathrm{sr}(A\otimes B)=1$ we conclude that $\mathrm{sr}(A\otimes B)\neq 1$.
\end{proof}

\subsection{Reduction of the topological dimension}

Let $\mathcal{B}$ denote the class of C*-algebras that are stably isomorphic to commutative C*-algebras. Let $\mathcal{C}$ denote the class of C*-algebras that are stably isomorphic to 1-dimensional NCCW-complexes. Recall that a C*-algebra $A$ is a 1-dimensional noncommutative CW-complex (or shortly a NCCW-complex) if it can be expressed as a pullback diagram of the form:
\begin{align}\label{eq: CW}
\begin{aligned}
\xymatrix{A \ar[d] \ar[r] & E\ar[d]^\phi \\ \mathrm{C}_0([0,1], F)\ar[r]^{\,\,\quad\mathrm{ev}_0\oplus\mathrm{ev}_1} & F\oplus F}
\end{aligned}
\end{align}
where $E$ and $F$ are finite dimensional C*-algebras. (We do not assume that the map from $E$ to $F\oplus F$ is unital)

\begin{lemma}\label{lem: NCCW}
If $A$ is a 1-dimensional NCCW-complex and $I$ is an ideal of $A$ then $I$ can be written as an inductive limit of a sequence of 1-dimensional NCCW-complexes.
\end{lemma}
\begin{proof}
Let $A$ be a 1-dimensional NCCW-complex as in the pullback diagram \eqref{eq: CW} and let $F=\bigoplus_{i=1}^k\M_{n_i}$. Let $I$ be an ideal of $A$. Then by Theorem \ref{th: ideals} there are open subsets $(U_i)_{i=1}^k$ of $[0,1]$ and an ideal $E'$ of $E$ such that
 \begin{equation*}
\xymatrix{I \ar[d] \ar[r] & E'\ar[d]^{\phi\vert_{E'}} \\ \bigoplus_{i=1}^k\mathrm{C}_0(U_i, \M_{n_i})\ar[r]^{\,\,\,\quad\quad\mathrm{ev}_0\oplus\mathrm{ev}_1} & F\oplus F}
\end{equation*}
is a pullback diagram. Decompose each $U_i$ as the union of disjoint open subintervals of $[0,1]$. Denote by $V_i$ and $W_i$ the union of the open subintervals  in the decomposition of $U_i$ that contain either the point $0$ or the point $1$, and the union of the open subintervals in this decomposition that do not contain either point.  Let $B$ be the C*-algebra obtained by the pullback diagram
\begin{equation*}
\xymatrix{B \ar[d] \ar[r] & E'\ar[d]^{\phi\vert_{E'}} \\ \bigoplus_{i=1}^k\mathrm{C}_0(V_i, \M_{n_i})\ar[r]^{\,\,\,\quad\quad\mathrm{ev}_0\oplus\mathrm{ev}_1} & F\oplus F}.
\end{equation*}
Then $B$ is a 1-dimensional NCCW-complex and $A\cong B\oplus \bigoplus_{i=1}^k\mathrm{C}_0(W_i, \M_{n_i})$. For each $i$ write $W_i$ as the union of an increasing sequence of open subsets $W_i^{(j)}$, $j=1,2,\cdots$, where each $W_i^{(j)}$ is the disjoint union of a finite number of open intervals. It follows that $$B\oplus\bigoplus_{i=1}^k\mathrm{C}_0(W_i, \M_{n_i})=\varinjlim B\oplus\bigoplus_{i=1}^k\mathrm{C}_0(W_i^{(j)}, \M_{n_i}).$$
Now, each algebra $\mathrm{C}_0(W_i^{(j)}, \M_{n_i})$ is a 1-dimensional NCCW-complex. Therefore, $A$ can be written as an inductive limit of 1-dimensional NCCW-complexes.
\end{proof}

\begin{theorem}
Let $A$ be a C*-algebra that is locally contained in $\mathcal{B}$ and let $B$ be a simple inductive limit of 1-dimensional NCCW-complexes such that $\mathrm{K}_0(B)=\mathrm{K}_1(B)=0$. Then $A\otimes B$ is locally contained in $\mathcal{C}$. Moreover, if $A$ is separable then $A\otimes B$ can be written as an inductive limit of a sequence of C*-algebras in $\mathcal{C}$.
\end{theorem}
\begin{proof}
Since $B$ is an inductive limit of 1-dimensional NCCW-complexes $\mathrm{C}[0,1]\otimes B$ is an inductive limit of 2-dimensional NCCW-complexes. By Theorem \ref{th: stable rank} the stable rank of $\mathrm{C}[0,1]\otimes B$ is one. Also,  $\mathrm{K}_0(\mathrm{C}[0,1]\otimes B)=0$. Hence, by Theorem \cite[Theorem 4.2 (iv)]{Eilers-Loring-Pedersen} the C*-algebra  $\mathrm{C}[0,1]\otimes B$ is an inductive limit of 1-dimensional NCCW-complexes. By induction the same is true for the C*-algebras $\mathrm{C}([0,1]^n)\otimes B$, $n=1,2, \cdots$. Since $\mathrm{C}([0,1]^\N)=\varinjlim \mathrm{C}([0,1]^n)$ and 1-dimensional NCCW-complexes are semiprojective it follows that $\mathrm{C}([0,1]^\N)\otimes B$ is an inductive limit of 1-dimensional NCCW-complexes. Using semiprojectivity again we get that any quotient of $\mathrm{C}([0,1]^\N)\otimes B$ is the inductive limit of a sequence of 1-dimensional NCCW-complexes. 
If $X$ is a compact metric space then it can realized as a close subspace of $[0,1]^\N$. Hence, the C*-algebra $\mathrm{C}(X)\otimes B$ is isomorphic to a quotient of $\mathrm{C}([0,1]^\N)\otimes B$. This implies that $\mathrm{C}(X)\otimes B$ can be written as the inductive limit of a sequence of 1-dimensional NCCW-complexes. In particular, $\mathrm{C}(X)\otimes B\in \mathcal{C}$. By \cite[Theorem 3.3.5]{Engelking} every compact Hausdorff space is homeomorphic to an inverse limit of compact metric spaces. In particular, this implies that any abelian unital C*-algebra is locally contained in the class of abelian C*-algebras whose spectrum are compact metric spaces. It follows that $\mathrm{C}(X)\otimes B\in \mathcal{C}$ for every compact Hausdorff space $X$. Since the class $\mathcal{C}$ is closed under the operation of taking tensor products with matrix algebras and inductive limits we get $\mathrm{C}(X,\mathcal{K})\otimes B=\varinjlim\mathrm{C}(X,\M_n)\otimes B\in \mathcal{C}$. If $A$ is a hereditary subalgebra of $\mathrm{C}(X,\mathcal{K})$ then clearly $A\otimes B\in \mathcal{C}$. This concludes the proof of the first part of the theorem. (Note that if $X$ is locally compact and $\widetilde X$ denotes the one point compactification of $X$ then $\mathrm{C}_0(X, \K)$ is a hereditary subalgebra of $\mathrm{C}(\widetilde X,\K)$, thus  $\mathrm{C}_0(X, \K)\in \mathcal{C}$.)

Now assume that $A$ is separable. We have shown that $A\otimes B\in \mathcal{C}$. Note that this implies that $A\otimes\K\otimes B$ is locally contained in the class of C*-algebras that are ideals of 1-dimensional NCCW-complexes. Hence $A\otimes\K\otimes B$ is locally contained in the class of 1-dimensional NCCW-complexes by Lemma \ref{lem: NCCW}. Now using the semiprojectivity of 1-dimensional NCCW-complexes and the separability of $A\otimes\K\otimes B$ we can write $A\otimes\K\otimes B$ as an inductive limit of a sequence of 1-dimensional NCCW-complexes. The second part of the theorem follows using the following result proved in \cite[Corollary 4]{Coward-Elliott-Ivanescu}: If $A=\varinjlim A_n$ and the stable rank of $A$ is one then every hereditary subalgebra of $A$ can be written as inductive limit of hereditary subalgebras of the C*-algebras $A_n$.   
\end{proof}

\begin{corollary}
Let $A$ and $B$ be as in the theorem above. Then $\mathrm{dr}(A\otimes B)=1$.
\end{corollary}
\begin{proof}
The decomposition rank of a 1-dimensional NCCW-complex is at most 1 by \cite[Theorem 1.6]{Winter-subhom}. Hence by \cite[Corollary 3.11 and Proposition 3.10]{Kirchberg-Winter} the decomposition rank of a C*-algebra that is stable isomorphic to a 1-dimensional NCCW-complex is at most 1. Therefore, the decomposition rank of any C*-algebra in $\mathcal{C}$ is at most one. By the definition of decomposition rank and Averson's Extension Theorem the same holds for C*-algebras that are locally contained in $\mathcal{C}$. In particular, $\mathrm{dr}(A\otimes B)\le 1$. The decomposition rank of $A\otimes B$ can not be zero since this would imply that $A\otimes B$ is an AF-algebra (see \cite[Example 6.1 (i)]{Kirchberg-Winter}) contradicting the fact that $A\otimes B$ is projectionless. 
\end{proof}

\begin{corollary}
If $A$ is a C*-algebra that is locally contained in $\mathcal{B}$. Then $A\otimes \mathcal W$ is locally contained in $\mathcal{C}$. In particular, $\mathrm{dr}(A\otimes \mathcal W)=1$. 
\end{corollary}

\end{document}